\documentclass[a4paper]{amsart}
\usepackage{
mathtools,
amssymb,
amsthm,
xcolor,
enumitem,
mathrsfs,
}
\usepackage[hidelinks]{hyperref}
\numberwithin{equation}{section}

\newtheorem{theorem}{Theorem}[section]
\newtheorem{lemma}[theorem]{Lemma}
\newtheorem{proposition}[theorem]{Proposition}
\newtheorem{corollary}[theorem]{Corollary}

\theoremstyle{definition}
\newtheorem{definition}[theorem]{Definition}

\theoremstyle{remark}
\newtheorem{remark}[theorem]{Remark}

\newcommand{\R}{\mathbb{R}}
\newcommand{\N}{\mathbb{N}}
\let\mb\mathbb
\let\mc\mathcal
\let\ml\mathcal
\let\mr\mathrm

\let\ve\varepsilon
\let\vp\varphi
\let\ol\overline
\let\del\partial
\let\wto\rightharpoonup

\let\l\left
\let\r\right
\newcommand{\lv}{\lVert}
\newcommand{\rv}{\rVert}
\newcommand{\ce}{\mathrel{\mathop:}=} 
\renewcommand\Re{\operatorname{Re}}
\renewcommand\Im{\operatorname{Im}}
\newcommand{\til}{\widetilde}
\renewcommand{\hat}{\widehat}
\newcommand{\pt}{\partial}
\newcommand{\al}{\alpha}
\newcommand{\lam}{\lambda}
\newcommand{\om}{\omega}
\newcommand{\Del}{\Delta}
\newcommand{\rad}{\mathrm{rad}}

\title[2d-NLS with point interaction]{On stability and instability of standing waves for 2d-nonlinear Schr\"odinger equations with point interaction}
\date{\today}

\author[N. Fukaya]{
Noriyoshi Fukaya}

\address[N. Fukaya]{
Department of Mathematics,
Faculty of Science Division I,
Tokyo University of Science,
Tokyo, 162-8601, Japan}
\email{fukaya@rs.tus.ac.jp}

\author[V. Georgiev]{Vladimir Georgiev}
\address[V. Georgiev]{Dipartimento di Matematica Universit\`a di Pisa
                      Largo B. Pontecorvo 5, 56100 Pisa, Italy\\
                      and \\
                      Faculty of Science and Engineering \\ Waseda University \\
                      3-4-1, Okubo, Shinjuku-ku, Tokyo 169-8555 \\
                      Japan and IMI--BAS, Acad.
                      Georgi Bonchev Str., Block 8, 1113 Sofia, Bulgaria}
\email{georgiev@dm.unipi.it}

\author[M. Ikeda]{Masahiro Ikeda}
\address[M. Ikeda]{
Center for Advanced Intelligence Project, RIKEN, Japan/Department of Mathematics, Faculty of Science and Technology,
Keio University, 3-14-1 Hiyoshi, Kohoku-ku, Yokohama 223-8522, Japan}
\email{masahiro.ikeda@riken.jp/masahiro.ikeda@keio.jp}

\subjclass[2020]{35Q55; 35B35}

\keywords{
nonlinear Schr\"{o}dinger equation,
point interaction,
standing waves,
stability,
instability
}

\begin{document}
\maketitle

\begin{abstract}
We study existence and stability properties of ground-state standing waves for two-dimensional nonlinear Schr\"odinger equation with a point interaction and a focusing power nonlinearity. The Schr\"odinger operator with a point interaction $(-\Del_\al)_{\al\in\R}$ describes a one-parameter family of self-adjoint realizations of the Laplacian with delta-like perturbation. The operator $-\Del_\al$ always has a unique simple negative eigenvalue. We prove that if the frequency of the standing wave is close to the negative eigenvalue, it is stable. Moreover, if the frequency is sufficiently large, we have the stability in the $L^2$-subcritical or critical case, while the instability in the $L^2$-supercritical case.
\end{abstract}

    \section{Introduction}

We consider the following nonlinear Schr\"odinger equation (NLS) with a focusing power nonlinearity in two spatial dimension:
\begin{equation}\label{NLS}
  i \del_t u = -\Del_\al u - |u|^{p-1}u, \quad (t,x) \in \R \times \R^2,
\end{equation}
where $p > 1$ and $-\Del_\al$ is the Laplacian with a point interaction with strength $\al \in \R$ at the origin (see \eqref{eq:op_dom}--\eqref{eq:expG1} for the precise definition).

The study of the Laplace operator with point interactions in $\R^N$ ($N=1$, $2$, or $3$) seems to become intensive research area in the last decades. The first rigorous attempt to define and study the spectral properties of these operators was done in \cite{BF61} by Berezin and Faddeev in 1961. Their study was extended in the work \cite{AH81} by Albeverio and H\o{}egh-Krohn in 1981.

A simplest way to introduce a singular perturbation at a point is to consider the potential perturbation $-\Del+V_\varepsilon$ of the Laplace operator with regular potentials $V_\ve$ spiking up and shrinking around a point in the limit $\ve \to 0$. This approach is well-known for dimension $N=1$ \cite{AGKK84}, $N=2$ \cite{AGKH87}, and $N=3$ \cite{AGK82} (we also refer to \cite{AGKH05} for a comprehensive overview). In the case of one singular point, this approximation gives a self-adjoint operator $-\Del_\al$ depending on the parameter $\al \in \R$. To be more precise, starting with the symmetric operator of $-\Del|_{ C_c^\infty (\R^N\setminus\{0\}) }$, one can characterize all its nontrivial self-adjoint extension on $L^2(\R^N)$ by means of a parameter $\al\in\R$. A detailed overview of the construction and the main properties of $-\Del_{\al}$ can be found in \cite{CFN}. Fractional powers of the Laplace operator $-\Del_\al$ with singular perturbations are studied in \cite{GMS18, MS18}.

The dispersive properties of the Schr\"odinger group $(e^{it\Del_\al})_{t\in\R}$ have been studied intensively in the last years. In one-dimensional case, the dispersive and Strichartz estimates and boundedness wave operators have been studied, for example, in \cite{AN09, DMW11, HMZ07}. In higher dimensional case, such properties have been studied recently by \cite{CMY19, Y21} in $N=2$, \cite{DMSY18} in $N=3$. See also \cite{DPT06} for the weighted dispersive estimate in three dimension.

By using these dispersive properties, one can establish existence of nonlinear evolution flow. The local well-posedness for \eqref{NLS} (in two or three-dimensional cases) in the domain $D(-\Del_\al)$ has been studied in \cite{CFN}. See, e.g., \cite{AN09, DMW11} for the one-dimensional case.

Recently, the nonlinear dynamics around standing waves in one dimension are intensively studied in various contexts depending on whether the potential is attractive or repulsive. In the attractive regime, the orbital stability and instability of standing waves have been studied by \cite{FOO08, GHW04, KO09}, the asymptotic stability has been done in \cite{CM19, MMSas, MMS20}, and the strong instability has been done in \cite{OY16, FO19}. In the repulsive regime, the orbital stability and instability have been studied by \cite{FJ08, LFFKS08}, and the global dynamics below ground states were studied in \cite{II17}.

As related topics, we mention the NLS with the concentrated nonlinearity. See \cite{AFH21, BKKS08, KKS12} for $N=1$, \cite{ACCT20, ACCT21} for $N=2$, and \cite{AN13} for $N=3$.

However, much less is known about the nonlinear dynamics and the existence and stability/instability properties of the ground states for \eqref{NLS}. Up to our knowledge, there are no results treating such properties in the two-dimensional case. The main goal of the work is to cover this case and establish fundamental properties of ground states for \eqref{NLS} such as existence, uniqueness, nondegeneracy, and its orbital stability/instability.

To state our main results, we shortly give the definition of the operator $-\Del_\al$ in two dimension (see \cite[Chapter I.5]{AGKH05} for more details). The class of self-adjoint extensions in $L^2(\R^2)$ of the positive and densely defined symmetric operator $-\Del|_{C^\infty_0(\R^2\setminus\{0\})}$ is a one-parameter family of operators $(-\Del_\al)_{\al\in(-\infty,+\infty]}$. The extension $-\Del_{\al=\infty}$ is the Friedrichs extension and is precisely the free Laplacian on $L^2(\R^2)$ with domain $H^2(\R^2)$. All other extensions for $\al\in\R$ represent nontrivial operators with a point interaction at the origin, and are characterized explicitly by
\begin{gather}
\label{eq:op_dom}
  D(-\Del_\al) = \left\{  f + \frac{f(0)}{\beta_\al(\lam)} G_\lam \colon\, f\in H^2(\R^2) \right\},
\\\label{eq:opaction}
  (-\Del_\al+\lam)g = (-\Del+\lam)f \quad \text{for } g = f + \frac{f(0)}{\beta_\al(\lam)} G_\lam \in D(-\Del_\al).
\end{gather}
Here $\lam>0$ is a fixed constant with $\lam\ne-e_\al$ (see \eqref{eq:negaengen} for the definition of $e_\al$),
\begin{equation}\label{eq.2d1}
   \beta_\al(\lam) \ce \al + \frac{\gamma}{2\pi} + \frac{1}{2\pi} \ln \frac{\sqrt{\lam}}{2} ,
\end{equation}
$\gamma>0$ denoting Euler--Mascheroni constant, and $G_{\lam}$ is the Green function of $-\Del+\lam$ on $\R^2$, defined by the distributional relation $(-\Del+\lam)G_{\lam}=\delta$, that is
\begin{equation}\label{eq:expG1}
  G_{\lam}(x) \ce \frac1{2\pi} \ml{F}^{-1} \biggl[\frac{1}{|\xi|^2+\lam}\biggr](x) = \frac1{(2\pi)^2} \int_{\R^2} \frac{e^{ix\cdot\xi}}{|\xi|^2+\lam} \, d\xi,
\end{equation}
where $\ml{F}^{-1}$ is the inverse Fourier transform. Note that $f(0)$ makes sense for $f\in H^2(\R^2)$ from the embedding $H^2(\R^2)\hookrightarrow C(\R^2)$.

From the the expression~\eqref{eq:expG1}, we can easily check that
\begin{gather}
\label{eq:GlamH1ep}
  G_\lam \in H^{1-\ve}(\R^2),\quad G_\lam \notin H^1(\R^2),
\\\label{eq:GlamGmu}
  G_{\lam} - G_\mu \in H^{3-\ve}(\R^2), \quad G_{\lam} - G_\mu \notin H^{3}(\R^2) \quad \text{if $\lam \ne \mu$}
\end{gather}
for all $\ve > 0$ and $\lam, \mu > 0$. The function $G_\lam$ is also represented as
\begin{equation}\label{eq:defGlambda}
  G_\lam(x) = \frac{1}{2\pi} K_{0}(\sqrt{\lam}|x|),
\end{equation}
where $K_0$ is the modified Bessel function of the second kind (or the Macdonald function) of order zero. From the expression~\eqref{eq:defGlambda}, $G_\lam$ is positive, radial, and strictly decreasing function with
\begin{equation}\label{eq:Gexpdecay}
  G_\lam(r) \sim
  \left\{\begin{array}{@{}c l@{}}
   -\ln(\sqrt{\lam} r) & \text{as } r \to 0,
\\[\jot] r^{-1/2} e^{-\sqrt{\lam}r} & \text{as } r \to \infty
  \end{array} \right.
\end{equation}
(see \cite[Chapter 10]{NIST:DLMF} for more properties of $K_\nu$).

From the property \eqref{eq:GlamGmu} we can check that the definition of $-\Del_\al$ is independent of $\lam$. Indeed, if $g=f+f(0)\beta_\al(\lam)^{-1}G_\lam$ for some $f\in H^2(\R^2)$ and $\lam\ne-e_\al$, then for $\mu\ne-e_\al$, we have the decomposition
\[g = \til{f} + \frac{ \til{f}(0) }{ \beta_\al (\mu) } G_\mu , \quad \text{where } \til{f} \ce f + \frac{ f(0) }{ \beta_\al (\lam) } (G_\lam-G_\mu) \in H^2(\R^2),
  \]
and the relation
\[(-\Del_\al + \mu) \biggl( \til{f}+\frac{ \til{f}(0) }{\beta_\al(\mu)} G_\mu \biggr) = (-\Del + \lam) f + ( \mu - \lam ) \biggl( f+\frac{ f(0) }{ \beta_\al(\lam) } G_\lam \biggr).
  \]
This means the independence.

The spectral properties of $-\Del_\al$ are also known (see \cite[Theorem 5.4]{AGKH05}). The essential spectrum of $-\Del_\al$ are given by
\begin{align}\label{eq:spectrum1}
  &\sigma_{\mr{ess}} (-\Del_\al) = \sigma_{\mr{ac}}(-\Del_\al) = [0 , +\infty), &
  &\sigma_{\mr{sc}} (-\Del_\al) = \emptyset,
\end{align}
where $\sigma_{\mathrm{ac}}$ and $\sigma_{\mathrm{sc}}$ denote the set of the absolutely and singularly continuous spectrum, respectively. The operator $-\Del_\al$ has a simple negative eigenvalue
\begin{align}\label{eq:negaengen}
  &\sigma_{\mathrm{p}}(-\Del_\al) = \{e_\al\}, &
  &e_\al \ce -4e^{-4\pi\al-2\gamma}.
\end{align}
In this sense, $-\Del_\al$ can be regarded as an Schr\"odinger operator with an attractive potential. The normalized eigenfunction corresponding to the eigenvalue $e_\al$ is
\[\chi_\al\ce\frac{G_{-e_\al}}{\|G_{-e_\al}\|_{L^2}}.
  \]
Note that $\beta_\al(\lam)$ defined in \eqref{eq.2d1} is expressed as
\[\beta_\al(\lam) = \frac{1}{4\pi} \ln \frac{\lam}{-e_\al},
  \]
so one has
\[\beta_\al(\lam) > 0 \iff \lam > -e_\al.
  \]

The energy space (or the form domain) $H_{\al}^1(\R^2)$ associated with $-\Del_{\al}$ can be characterized by general results on the Kre{\u\i}n--Vi\v{s}ik--Birman extension theory (see e.g.~\cite{MOS-2018FP}), and it is given explicitly by
\begin{equation}\label{eq:energyspace}
  H_{\al}^1(\R^2) =
  \left\{\begin{array}{@{}c l@{}}
  \{f+cG_\lam\colon\, f\in H^1(\R^2),\ c\in\mb{C}\} & \text{if } \al \in \R,
\\[\jot]
  H^1(\R^2) & \text{if } \al = \infty.
  \end{array}\right.
\end{equation}
Note that $H_\al^1(\R^2)$ is independent of the choice of $\al\in\R$ and $\lam>0$ from \eqref{eq:GlamGmu}. For $\lam>-e_\al$, one can define the maximal extension of the form
\[\langle ( -\Del_{\al} + \lam ) g, g \rangle = \| \nabla f \|_{L^2}^2 + \lam \| f \|_{L^2}^2 + \frac{|f(0)|^2}{ \beta_\al (\lam) }
  \]
with $g = f + f(0) \beta_\al (\lam)^{-1} G_\lam \in D (-\Del_{\al})$. This extension defines a positive quadratic form well-defined on $g\in H_{\al}^1(\R^2)$ explicitly given by
\begin{equation}\label{eq:qf}
  \langle (-\Del_{\al} + \lam) g, g \rangle = \| \nabla f \|_{L^2}^2 + \lam \|f\|_{L^2}^2 + \beta_\al(\lam) |c|^2
\end{equation}
for $g=f+cG_\lam\in H_\al^1(\R^2)$. By using the relation
\[\lam\|g\|_{L^2}^2 = \lam\| f + cG_\lam \|_{L^2}^2 = \lam\|f\|_{L^2}^2 + 2 \lam \Re [ c(f,G_\lam)_{L^2} ] + \frac{|c|^2}{4\pi},
  \]
we can rewrite \eqref{eq:qf} as
\[\langle -\Del_{\al} g, g \rangle = \| \nabla f \|_{L^2}^2 - 2 \lam \Re [ c(f,G_\lam)_{L^2} ] + \left( \beta_\al(\lam) - \frac{1}{4\pi} \right)|c|^2.
  \]

We denote the $H_\al^1$-norm depending on the parameter $\lam>-e_\al$ by
\begin{equation} \label{nn21}
  \| g \|_{H_{\al,\lam}^1} \ce \sqrt{ \langle (-\Del_\al+\lam) g, g \rangle }
\end{equation}
for $g\in H_{\al}^1(\R^2)$. The following equivalence holds:
\[\| g \|_{H_{\al,\lam}^1} \simeq \| g \|_{H_{\al,\mu}^1}, \quad \lam, \mu > -e_\al.
  \]
As a special case, we choose $\lam = 1 - e_\al$ and also use the notation
\[\| g \|_{H_\al^1} \ce \| g \|_{H_{\al,1-e_\al}^1}.
  \]

To study the stability properties of the standing waves, we need the local well-posedness in the energy spaces $ H^1_\al(\R^2)$.
We have the following statement (see Appendix~\ref{sec:B} for the proof).
\begin{proposition}\label{Prop:LWP}
Let $\al\in\R$ and $p>1$. For each $u_0\in H_\al^1(\R^2)$, there exists the unique maximal solution
\[u \in C \bigl( (-T_{\min}, T_{\max}), H_\al^1(\R^2) \bigr) \cap C^1 \bigl( (-T_{\min}, T_{\max}), H_\al^{-1}(\R^2) \bigr)
  \]
of \eqref{NLS} with the initial data $u(0)=u_0$, where $H_\al^{-1}(\R^2)$ is the dual space of $H_\al^1(\R^2)$. Moreover, $u$ satisfies the conservation laws of energy and the $L^2$-norm:
\begin{align*}
  &E(u(t)) = E(u_0),&
  &\|u(t)\|_{L^2} = \|u_0\|_{L^2}
\end{align*}
for all $t\in(-T_{\min},T_{\max})$, where the energy is defined by
\[E(v) \ce \frac{1}{2} \langle -\Del_\al v, v \rangle
  -\frac{1}{p+1} \| v \|_{ L^{p+1} }^{p+1}, \quad v \in H_\al^1 (\R^2). \]
\end{proposition}

Now let us consider standing wave solutions with the form
\begin{equation}
  u(t,x) = e^{i \om t} \phi(x),
\end{equation}
where $\phi\in H_\al^1(\R^2)$ is a nontrivial solution of the stationary equation
\begin{equation}\label{eq:sp}
  -\Del_\al\phi + \om\phi - |\phi|^{p-1} \phi = 0, \quad x \in \R^2.
\end{equation}

Equation~\eqref{eq:sp} can be rewritten as $S_{\om}'(\phi)=0$, where $S_\om$ is the action functional defined by
\begin{equation}\label{eq:action}
  S_{\om} (v) = S_{\al,\om} (v) \ce \frac12\| v \|_{H_{\al,\om}^1}^2 - \frac1{p+1} \| v \|_{ L^{p+1} }^{p+1}, \quad v \in H_\al^1 (\R^2).
\end{equation}
We denote the set of all nontrivial solution of \eqref{eq:sp} by
\begin{align*}
  \mc{A}_{\om} = \mc{A}_{\al, \om} &\ce \{ \phi \in H_\al^1 (\R^2) \colon\, \phi \ne 0,\ S_{\al,\om}' (\phi) = 0 \}
\end{align*}
and the set of all ground states (minimal action solution) by
\begin{align*}
  \mc{G}_{\om} = \mc{G}_{\al, \om} & \ce \{ \phi \in \mc{A}_{\al,\om} \colon\, S_{\al,\om} (\phi) \le S_{\al,\om} (\psi) \text{ for all }\psi \in \ml{A}_{\al,\om} \}.
\end{align*}

Now we state our main results. First, we state the results about the existence, symmetry, and uniqueness of ground states.
\begin{theorem}\label{thm:exisG}
If $\om>-e_\al$, then the set $\ml{G}_{\om}$ is not empty.
\end{theorem}

\begin{theorem}
\label{thm:symmetryGS}
Let $\om>-e_\al$ and $\phi\in\ml{G}_\om$. If $\phi$ is decomposed as $\phi=f+f(0)\beta_\al(\om)^{-1}G_\om$ with $f\in H^2(\R^2)$, then there exists $\theta\in\R$ such that $e^{i\theta}f$ is positive, radial, and decreasing function. In particular, the function $e^{i\theta}\phi$ is also positive, radial, and decreasing.
\end{theorem}

\begin{theorem} \label{thm:uniqG}
There exists $\om_1 > -e_\al$ such that if $\om > \om_1$, then there exists the unique positive radial ground state $\phi_\om \in \ml{G}_\om$ such that the set of all ground states is characterized as
\begin{equation}\label{eq:uniqG}
  \ml{G}_{\om} = \{ e^{i \theta} \phi_\om \colon\, \theta \in \R \}.
\end{equation}
\end{theorem}

Next, we state the results about orbital stability and instability of standing waves. The definition of orbital stability is as follows.

\begin{definition}
Let $\om\in\R$ and let $\phi\in H_\al^1(\R^2)$ be a nontrivial solution of \eqref{eq:sp}. The standing wave $e^{i\om t}\phi$ is \emph{stable} if for any $\ve>0$ there exists $\delta>0$ such that for any $u_0\in H_\al^1(\R^2)$ satisfying $\|u_0-\phi\|_{H_\al^1}<\delta$, the solution $u(t)$ of \eqref{NLS} with $u(0)=u_0$ exists globally in time and satisfies
\[\sup_{t\in\R}\inf_{\theta\in\R}\|u(t)-e^{i\theta}\phi\|_{H_\al^1}
  <\ve. \]

Otherwise, the standing wave $e^{i\om t}\phi$ is \emph{unstable}.
\end{definition}

The following two statements are main results of this paper. The first one concerns the stability for $\om$ close to $-e_\al$.

\begin{theorem}\label{thm:stabsmall}
For each $\al\in\R$ and $p>1$, there exists $\om_* > -e_\al$ such that if $\om\in(-e_\al,\om_*)$ and $\phi\in\ml{G}_\om$, the standing wave $e^{i\om t}\phi$ is stable.
\end{theorem}

The second one concerns the stability/instability for large frequency $\om$. We denote the unique ground state given in Theorem~\ref{thm:uniqG} by $\phi_{\om}$.

\begin{theorem}\label{thm:stablarge}
For each $\al\in\R$ and $p>1$, there exists $\om^*\in(\om_1,\infty)$ such that the following is true.
\begin{itemize}
\item If $1<p\le 3$, then the standing wave $e^{i\om t}\phi_\om$ is stable for all $\om>\om^*$.
\item If $p>3$, then the standing wave $e^{i\om t}\phi$ is unstable for all $\om>\om^*$.
\end{itemize}
\end{theorem}

One can observe the similarity between the results in \cite{Fs, F05, FO03I, FO03S, FOO08} and Theorems~\ref{thm:stabsmall} and \ref{thm:stablarge}. The paper \cite{FOO08} treats NLS with attractive $\delta$-potential in one dimension, and the papers \cite{Fs, F05, FO03I, FO03S} concern NLS with general attractive potential $V(x)$. Since $-\Del_\al$ has a unique simple negative eigenvalue, we regard it as a Schr\"odinger operator with an attractive potential, so it is natural to choose to follow the approach in these papers.

Let us give a short outline of the proofs. The local well-posedness (Proposition~\ref{Prop:LWP}) follows from the energy methods in \cite{OSY12} (see also \cite[Chapter~3]{C03}) and the Strichartz estimates for the operator $-\Del_\al$ obtained by \cite{CMY19}. The existence of ground states (Theorem~\ref{thm:exisG}) follows from a standard variational method by using the Nehari manifold. The positivity and symmetry of ground states (Theorem~\ref{thm:symmetryGS}) follow from the maximal principle and the symmetric rearrangement. In particular, we use the result of Brothers and Ziemer \cite{BZ88} to obtain the radial symmetry and decrease of ground states.

To investigate the stability properties, we consider rescaled ground states and use a perturbation argument as in \cite{F05, FO03I, FO03S}. Let $(\phi_\om)_{\om>-e_\al}$ be a family of positive ground states with $\phi_\om=\phi_{\al, \om}\in\ml{G}_{\al, \om}$. For $\om$ close to $-e_\al$, we normalize the ground states as
\[\hat\phi_\om(x) \ce \frac{\phi_\om}{\lv \phi_\om \rv_{L^2} }. \]
Then $\hat\phi_\om$ is a positive solution of
\[-\Del_\al \hat\phi + \om \hat\phi - \lv \phi_\om \rv_{L^2}^{p-1} |\hat\phi|^{p-1}\hat\phi = 0.
  \]
Since we can verify $\lv \phi_\om \rv_{L^2} \to 0$ as $\om \to -e_\al$, $\hat\phi_\om$ can be regarded as a perturbation of the solution for the linear equation
\begin{equation}\label{eq:lineq}\mathopen{}
  -\Del_\al\hat\phi
  =e_\al\hat\phi,
\end{equation}
that is, $\hat\phi_\om$ is close to the eigenfunction $\chi_\al$. On the other hand, for large $\om$, we rescale $\phi_\om$ as
\[\til\phi_\om(x)
\ce\om^{-1/(p-1)}\phi_\om(x/\sqrt{\om}). \]
Then $\til\phi_\om$ is a positive solution of
\[-\Del_{\til{\al}} \til\phi + \til\phi - |\til\phi|^{p-1} \til\phi = 0,
  \]
where
\[\til\al = \til\al(\om) \ce \al + \frac{1}{4\pi} \ln \om.
  \]
Since $-\Del_{\al=\infty}$ is the free Laplacian $-\Del$, $\til\phi_\om$ can be regarded as a perturbation of the solution for the equation
\begin{equation}\label{eq:nonlineq}
  {-}\Delta \til{\phi} + \til\phi - \lvert \til\phi \rvert^{p-1} \til\phi = 0.
\end{equation}
Therefore, we can investigate the stability properties by using the limiting equation \eqref{eq:lineq} for small $\om$ and \eqref{eq:nonlineq} for large $\om$.

The stability for small frequency (Theorem~\ref{thm:stabsmall}) follows from the argument of \cite{FO03S}. If $\om$ is sufficiently close to $-e_\al$, we can obtain the following coercivity property for the linearlized operator around the ground state.

\begin{proposition}\label{prop:coercivity}
For each $\al\in\R$ and $p>1$, there exists $\om_* > -e_\al$ such that if $\om\in(-e_\al,\om_*)$ and $\phi\in\ml{G}_\om$, the following holds: There exists a positive constant $k$ such that
\[\langle S_{\om}''(\phi) w, w \rangle \ge k\| w \|_{H_\al^1}^2
  \]
for any $w \in H_\al^1 (\R^2)$ satisfying $\int_{\R^2} \phi \ol{w}\, dx = 0$.
\end{proposition}

It is known that this coercivity implies the stability (see, e.g., \cite{GSS87, L09}). Because we can prove Proposition~\ref{prop:coercivity} exactly in the same way as \cite[Section~4]{FO03S}, we omit the proof.

\begin{remark}
In this paper, we do not discuss the uniqueness of ground states for small frequencies because we do not need it if we just prove the stability. However, we can obtain the uniqueness by the bifurcation theory. See, e.g., \cite{RW88, GNT04, MMS20} for more details.
\end{remark}

To investigate the properties of the ground states for large frequency $\om$, we use the limiting equation \eqref{eq:nonlineq}. It is well known that \eqref{eq:nonlineq} has the unique positive radial ground state $\til\phi_\infty\in H^1(\R^2)$ (see, e.g., \cite{BL83I,K89}), and it is nondegenerate in the radial space, that is, the kernel of the linearized operator $\til L_\infty \ce -\Del+1-p\til\phi_\infty^{p-1}$ is trivial: $\ker\til L_\infty|_{H_{\rad}^1}=\{0\}$. By using these properties, we establish the uniqueness (Theorem~\ref{eq:uniqG}) and nondegeneracy (Lemma~\ref{lem:ftL}) for large $\om$ following the argument of \cite[Proposition~2 (v)]{F05}. Moreover, we can obtain the regularity of the map $\om\mapsto\phi_\om$ (Corollary~\ref{cor:reg}). To obtain the stability and instability, we use the following criteria.

\begin{proposition}[\cite{O14, S83}] \label{prop:scstabinsta}
For $\om>\om_1$, the standing wave $e^{i\om t}\phi_\om$ of \eqref{NLS} is stable if $\frac{d}{d\om}\|\phi_\om\|_{L^2}^2>0$ and unstable if $\frac{d}{d\om}\|\phi_\om\|_{L^2}^2<0$.
\end{proposition}

\begin{remark}
Proposition~\ref{prop:scstabinsta} are well-known as the criteria of Grillakis, Shatah, and Strauss~\cite{GSS87} (see also \cite{SS85, W86}). To use their result, we need to investigate the spectral properties of the linearized operator $S_\om''(\phi_\om)$, but we do not discuss its spectra in this paper. Instead, we can apply the arguments of \cite{S83} for the stability and \cite{O14} for the instability because they only require the variational characterization on the Nehari manifold, the uniqueness, and the differentiability of the map $\om\mapsto\phi_\om$ with $\phi_\om\in\ml{G}_\om$. These properties are discussed in this paper.
\end{remark}

From Proposition~\ref{prop:scstabinsta}, the stability/instability problems can be reduced to the investigation of the sign of the derivative $\frac{d}{d\om}\|\phi_\om\|_{L^2}^2$. When $\al=\infty$, i.e., without interaction, one can show by the scaling invariance for the equation that the ground states $\phi_{\infty,\om}$ of \eqref{eq:sp} satisfies $\frac{d}{d\om} \| \phi_{\infty,\om} \|_{L^2}^2 > 0$ if $1 < p < 3$ and $\frac{d}{d\om} \| \phi_{\infty,\om} \|_{L^2}^2 < 0$ if $p > 3$ for all $\om > 0$. This means that when $\al=\infty$, the ground-state standing wave $e^{i \om t}\phi_{\infty,\om}$ of \eqref{NLS} is stable if $1<p<3$ \cite{CL82} and unstable if $p \ge 3$ \cite{BC81} (see \cite{W82} for $p=3$).

To investigate the sign of $\frac{d}{d\om}\|\phi_\om\|_{L^2}^2$ for large $\om$, we apply the argument of \cite{F05, Fs}. Instead of $\frac{d}{d\om}\|\phi_\om\|_{L^2}^2$, we calculate the rescaled version $\frac{d}{d\om}\|\til\phi_\om\|_{L^2}^2$ and use the convergence $\til\phi_\om\to \til\phi_\infty$. To estimate some error terms, we establish and use a boundedness of the inverse linearized operator of $\til\phi_\om$. After that, we can determine the sign of $\frac{d}{d\om}\|\phi_\om\|_{L^2}^2$, and combining Proposition~\ref{prop:scstabinsta} we obtain Theorem~\ref{thm:stablarge}.

The difficulty of the proofs of our results mainly comes from the treatment of functions in the energy space $H_\al^1(\R^2)$ and the domain $D(-\Del_\al)$. Of special importance in one-dimensional case is the fact that the fundamental solution of $(1-\Del)$ is in $H^1(\R)$, so one can use $H^1(\R)$ as a natural space of the nonlinear flow associated with the corresponding NLS. The situation changes essentially in two dimension since we are forced to work with the perturbed $H_\al^1(\R^2)$ space, so there are nontrivial difficulties to apply of the variational technique from \cite{FOO08} and the cases of slowly decaying potentials \cite{Fs, F05, FO03I, FO03S}. For a function in the spaces $H_\al^1(\R^2)$ or $D(-\Del_\al)$, we need to decompose it into the regular and singular parts and to treat these separately, and we have to avoid several difficult points requiring appropriate new treatments.
\begin{itemize}
\item The local well-posedness for the standard 2d NLS with or without potential requires the use of Strichartz estimates in Sobolev spaces
\begin{equation}\label{eq.lwp1}
  H^{s,p}(\R^2) = \{ g = (1 - \Del)^{-s} \psi \colon\, \psi \in L^p(\R^2) \}, \quad p \in (1,2), \ s\in (0,1],
\end{equation}
if a contraction argument is applied. On the other hand,  the case of singular perturbed Laplacian $-\Del_\al$ requires the replacement of the classical Sobolev space $H^1(\R^2)$ by the perturbed space $H^1_\al(\R^2)$, and we need to decompose functions $g \in H^1_\al(\R^2)$ as
\begin{equation}\label{eq:domdecom}
  g = f + c G_\lam.
\end{equation}
There is no flexible treatment (up to our knowledge) of appropriate generalization of generic spaces $H^{s,p}$ for the Laplacian of type $-\Del_\al$. For this we have chosen another approach based on compactness argument and the results in \cite{OSY12}.
\item The existence of ground states  seems to be closely connected with the inclusion
\[H^1(\R^2) \subset H^1_{\al}(\R^2).
  \]
However, the ground states $\phi = f + f(0) \beta_\al(\om)^{-1} G_\om$ from Theorem \ref{thm:symmetryGS} have nontrivial singular part, since $e^{i\theta} f(0)$ is positive. This fact shows that the ground states associated with $\al \in \R$ are different from ground states with $\al=\infty$. Moreover the ground state $\phi$ from Theorem \ref{thm:symmetryGS} is not in $H^1(\R^2)$.
\item The symmetry of the ground state for the classical NLS can be obtained by Schwartz symmetrization. Since we have the decomposition $\phi = f + f(0) \beta_\al(\om)^{-1} G_\om$ for any $\phi \in D(-\Delta_\al)$ into regular and singular parts, a formal symmetrization
\[\phi^* = \left( f + \frac{f(0)}{\beta_\al(\lam)} G_\lam \right)^*
  \]
cannot work. We have chosen the following symmetrization
\[ f + \frac{f(0)}{\beta_\al(\lam)} G_\lam \to f^* + \frac{|f(0)|}{\beta_\al(\lam)} G_\lam. \]
The technical difficulties associated with this symmetrization can be overcome by using the results in \cite{AL89}.
\item The uniqueness and nondegeneracy of ground states require careful use of the decomposition \eqref{eq:domdecom} and modify accordingly the approach in \cite{FO03I, FO03S, F05}.
\item To determine the sign of $\frac{d}{d\om}\|\til\phi_\om\|_{L^2}^2$, we need to estimate the error term $\pt_\om\til{f}_\om(0)$ coming from the interaction of $-\Del_\al$, where $\til{f}_\om$ is the regular part of $\til{\phi}_\om$. As in the previous work \cite{Fs,F05}, we use the boundedness of the inverse of the linearized operator $\til{L}_{\om}$. However, it is not trivial how to express and estimate the term $\pt_\om\til{f}_\om(0)$ by using the operator $\til{L}_{\om}^{-1}$. To overcome this difficulty we make a good use the expression \eqref{eq:opaction} of the operator and the expression of the bilinear form as
\[f(0)
  =\left\langle (\Delta_\al+\lam)G_\lam, f + \frac{f(0)}{\beta_\al(\lam)} G_\lam \right\rangle \]
for $f\in H^2(\R^2;\R)$. For more details, see Lemma~\ref{lem:deltilf0}.
\end{itemize}

The rest of organization of this paper is as follows. In Section~\ref{sec:2} we correct the properties of $G_\lam$ used in this paper. In Section~\ref{sec:3} we prove Theorem~\ref{thm:exisG} through the characterization with the Nehari functional. Section~\ref{sec:4} is devoted to the proof of Theorem~\ref{thm:symmetryGS}. In Section~\ref{sec:5} we show that a family of rescaled ground states converges to the ground state of NLS without interaction (i.e. $\al=\infty$) as $\om\to\infty$. In Section~\ref{sec:6} we show lower boundedness and nondegeneracy of the linearized operator around the ground state for large $\om$. This lower boundedness will be used in Section~\ref{sec:9} as the boundedness of the inverse operator. In Section~\ref{sec:7} we prove Theorem~\ref{eq:uniqG}. In Section~\ref{sec:8} we discuss the regularity of the map $\om\mapsto\phi_\om$ for large $\om$. Finally we prove Theorem~\ref{thm:stablarge} in Section~\ref{sec:9}. In Appendix~\ref{sec:A} we review the properties of wave operators and Strichartz estimates for the operator $-\Del_\al$. In Appendix~\ref{sec:B} we prove Proposition~\ref{Prop:LWP}.



    \section{Preliminaries} \label{sec:2}

The aim of this section is to recall the main properties of the singular-perturbed Laplace operator $-\Del_\al$ and the Green function $G_\lam$.


Note that \eqref{eq:GlamH1ep} and \eqref{eq:Gexpdecay} imply
\begin{equation}\label{eq.pprlm1}
  G_\lam \in L^q(\R^2) \quad \text{for all } q \in [1, \infty ).
\end{equation}
This fact leads easily to the following Sobolev inequality

\begin{lemma}\label{lem:Sobemb}
For any $q \in [2, \infty)$ and $\lam > -e_\al$ there exists a constant $C > 0$ such that
\begin{equation}\label{eq.sob1}
   \| v \|_{L^q} \le C \| v \|_{H_{\al,\lam}^1}
\end{equation}
for all $v \in H_{\al}^1 (\R^2)$.
\end{lemma}

\begin{proof}
Any $v \in H_{\al}^1 (\R^2)$ has the representation $v = f + c G_\lam$ for some $\lam > -e_\al$, $f \in H^1 (\R^2)$, and $c \in \mb{C}$. Then the property \eqref{eq.pprlm1} implies that
\[\| v \|_{L^q} \lesssim \|f\|_{L^q} + |c| \lesssim \| f \|_{H^1} + |c|. \]
The relation \eqref{nn21} implies that
\[ \| f \|_{H^1} + |c| \sim \| v \|_{ {H_{\al,\lam}^1} }. \]
Hence, we have \eqref{eq.sob1}.
\end{proof}

\begin{lemma}
For $\lam,\mu>0$, the inner product of $G_\lam$ and $G_\mu$ is given by
\[( G_\lam, G_\mu )_{L^2} = \left\{\begin{array}{@{} c l @{}}
  \dfrac{ \ln\lam - \ln\mu }{ 4 \pi ( \lam - \mu ) }, & \mu \ne \lam,
\\[3\jot]
  \dfrac{1}{ 4 \pi \lam}, &\mu = \lam.
  \end{array}\right.
  \]
\end{lemma}

\begin{proof}
The assertion follows from direct calculations.
\end{proof}

\begin{lemma}\label{lem:xgG1}
$\ml{F} [ x \cdot \nabla G_1 ] = -2 ( |\xi|^2 + 1 )^{-1} \ml{F} [G_1]$.
\end{lemma}

\begin{proof}
By a direct calculation, we have
\begin{align*}
  \nabla \ml{F} [G_1] & = -\frac{\xi}{\pi(|\xi|^2+1)^2},
\\\ml{F} [ x \cdot \nabla G_1 ] &= -2 \ml{F} [G_1] - \xi \cdot \nabla \ml{F}[G_1] = -\frac{1}{ \pi ( |\xi|^2 + 1 ) } + \frac{ |\xi|^2 }{ \pi ( |\xi|^2 + 1 )^2 }
\\&= -\frac{1}{ \pi ( |\xi|^2 + 1 )^2} = -2 ( |\xi|^2 + 1 )^{-1} \ml{F} [G_1].
\end{align*}
Thus, we have the assertion.
\end{proof}

    \section{Existence of ground states}\label{sec:3}

In this section, we prove existence of ground states for \eqref{eq:sp} by using a standard variational method and properties of the operator $-\Del_\al$. Throughout this section, we fix $\om>-e_\al$. We define the Nehari functional by
\begin{align*}
  K_{\al,\om}(v)
  &\ce\del_\lam S_{\al,\om}(\lam v)|_{\lam=1}
\\&=\langle S_{\al,\om}'(v),v\rangle
  =\|v\|_{H_{\al,\om}^1}^2
  -\|v\|_{L^{p+1}}^{p+1}
\end{align*}
for $v\in H_\al^1(\R^2)$. We denote
\begin{align*}
  \mc{K}_{\al,\om}
  &\ce\{v\in H_\al^1(\R^2)\colon\, v\ne0,~K_{\al,\om}(v)=0\},
\\d_\al(\om)
  &\ce\inf\{S_{\al,\om}(v)\colon\, v\in\mc{K}_{\al,\om}\},
\\\mc{M}_{\al,\om}
  &\ce\{v\in\mc{K}_{\al,\om}\colon\, S_{\al,\om}(v)=d_\al(\om)\}.
\end{align*}
For simplicity of notations, we shall often omit the subscript $\al$ like $S_{\om}$, $K_{\om}$, and so on. We note that $\ml{G}_{\om}\subset\ml{A}_{\om}\subset\ml{K}_{\om}$.

We will prove the following.
\begin{proposition}\label{prop:eg}
For any $\om>-e_\al$,
\begin{align*}
  \mc{G}_{\om}
  =\mc{M}_{\om}
  \ne\emptyset.
\end{align*}
\end{proposition}

By using the functional $K_{\om}$, we can rewrite the action as
\begin{align}\label{eq:SHalK}
  S_{\om}(v)
  &=\frac{p-1}{2(p+1)}\|v\|_{H_{\al,\om}^1}^2
  +\frac{1}{p+1}K_{\om}(v)
\\ \label{eq:SLpK}
  &=\frac{p-1}{2(p+1)}\|v\|_{L^{p+1}}^{p+1}
  +\frac{1}{2}K_{\om}(v),
\end{align}
and $d_\al(\om)$ as
\begin{align}
  \label{eq:variH1al}
  d_\al(\om)
  &=\inf\left\{\frac{p-1}{2(p+1)}\|v\|_{H_{\al,\om}^1}^2\colon\, v\in\mc{K}_{\om}\right\}
\\\label{eq:variLp}
  &=\inf\left\{\frac{p-1}{2(p+1)}\|v\|_{L^{p+1}}^{p+1}\colon\,
  v\in\mc{K}_{\om}\right\}.
\end{align}

\begin{lemma}\label{lem:mg}
$\mc{M}_{\om}\subset \mc{G}_{\om}$.
\end{lemma}

\begin{proof}
Let $\phi\in\mc{M}_{\om}$. By $K_{\om}(\phi)=0$, we have
\begin{equation} \label{eq:k'}
  \langle K_{\om}'(\phi),\phi\rangle
  =2\|\phi\|_{H_{\al,\om}^1}^2-(p+1)\|\phi\|_{L^{p+1}}^{p+1}
  =-(p-1)\|\phi\|_{L^{p+1}}^{p+1}
  <0. \end{equation}
Therefore, there exists a Lagrange multiplier $\eta\in\R$ such that $S_{\om}'(\phi)=\eta K_{\om}'(\phi)$. Moreover, since
\[\eta\langle K_{\om}'(\phi),\phi\rangle
   =\langle S_{\om}'(\phi),\phi\rangle
   =K_{\om}(\phi)=0, \]
it follows from \eqref{eq:k'} that $\eta=0$, which implies $S_{\om}'(\phi)=0$.

Furthermore, if $\psi\in\mc{A}_{\om}$, by $\phi\in\ml{M}_{\om}$ and $\psi\in\ml{K}_{\om}$, we have $S_{\om}(\phi)\le S_{\om}(\psi)$. Thus, we obtain $\phi\in\mc{G}_{\om}$. This completes the proof.
\end{proof}

\begin{lemma} \label{lem:knega}
If $v\in H_\al^1(\R^2)$ satisfies $K_{\om}(v)<0$, then
\begin{align*}
  &\frac{p-1}{2(p+1)}\|v\|_{L^{p+1}}^{p+1}>d_\al(\om),&
  &\frac{p-1}{2(p+1)}\|v\|_{H_{\al,\om}^1}^2>d_\al(\om).
\end{align*}
\end{lemma}

\begin{proof}
Let $v\in H_{\al}^1(\R^2)$ satisfy $K_{\om}(v)<0$. From the shape of the graph of the function $\lam\mapsto K_{\om}(\lam v)=\lam^2\|v\|_{H_{\al,\om}^1}^2-\lam^{p+1}\|v\|_{L^{p+1}}^{p+1}$, there exists $\lam_0\in(0,1)$ such that $K_{\om}(\lam_0 v)=0$. From \eqref{eq:variH1al}, we obtain
\[d_\al(\om)
  \le \frac{p-1}{2(p+1)}\|\lam_0 v\|_{L^{p+1}}^{p+1}
  =\frac{p-1}{2(p+1)}\lam_0^{p+1}\|v\|_{L^{p+1}}^{p+1}
  <\frac{p-1}{2(p+1)}\|v\|_{L^{p+1}}^{p+1}. \]
Similarly, by using \eqref{eq:variLp} we have $d_\al(\om)<\frac{p-1}{2(p+1)}\|v\|_{H_{\al,\om}^1}^2$.
\end{proof}

We note that from Lemma~\ref{lem:knega}, the expression~\eqref{eq:variH1al} can be written as
\begin{equation}\label{eq:variH1al2}
  d_\al(\om)
  =\inf\left\{\frac{p-1}{2(p+1)}\|v\|_{H_{\al,\om}^1}^2\colon\,
  v\in H_\al^1(\R^2),\ v\ne 0,\ K_{\om}(v)\le 0\right\}.
\end{equation}

\begin{lemma}\label{lem:posimu}
$d_{\al}(\om)>0$.
\end{lemma}

\begin{proof}
Let $v\in\mc{K}_{\om}$. From $K_{\om}(v)=0$ and the embedding $H_\al^1(\R^2)\hookrightarrow L^{p+1}(\R^2)$ (Lemma~\ref{lem:Sobemb}), we have
\begin{align*}
  \|v\|_{H_{\al,\om}^1}^2
  =\|v\|_{L^{p+1}}^{p+1}
  \lesssim \|v\|_{H_{\al,\om}^1}^{p+1}. \end{align*}
Since $v\ne0$ and $p>1$, we have the uniform bound $\|v\|_{H_{\al,\om}^1}\gtrsim 1$. From the expression \eqref{eq:variH1al} we obtain the conclusion.
\end{proof}

Now we use the action and Nehari functional without potential defined by
\begin{align*}
  S_{\infty,\om}(f)
  &\ce\frac12\|\nabla f\|_{L^2}^2
  +\frac{\om}{2}\|f\|_{L^2}^2
  -\frac{1}{p+1}\|f\|_{L^{p+1}}^{p+1},
\\K_{\infty,\om}(f)
  &\ce\|\nabla f\|_{L^2}^2
  +\om\|f\|_{L^2}^2
  -\|f\|_{L^{p+1}}^{p+1},
\end{align*}
respectively. We denote the minimal action value without potential by
\begin{align*}
  d_\infty(\om)
  \ce\min\{S_{\infty,\om}(v)\colon\,
  v\in H^1(\R^2),\ v\ne 0,\ K_{\infty,\om}(v)=0\}.
\end{align*}
It is known that there exists the unique positive radial ground state $\phi_{\infty,\om}\in H^1(\R^2)$ of the equation
\[-\Del\phi
  +\om\phi
  -|\phi|^{p-1}\phi
  =0,\quad
  x\in\R^2, \]
and that $\phi_{\infty,\om}$ satisfies $S_{\infty,\om}(\phi_{\infty,\om})=d_\infty(\om)$.

We note that
\begin{align*}
  &S_{\al,\om}(f)
  =S_{\infty,\om}(f),&
  &K_{\al,\om}(f)
  =K_{\infty,\om}(f)
\end{align*}
for all $f\in H^1(\R^2)$, which implies $d_\al(\om)\le d_\infty(\om)$. We prove the strict inequality.

\begin{lemma}\label{lem:muineq}
$d_{\al}(\om)<d_\infty(\om)$.
\end{lemma}

\begin{proof}
Suppose that $d_{\al}(\om)=d_\infty(\om)$. Then we have
\begin{align*}
  &S_{\al,\om}(\phi_{\infty,\om})
  =S_{\infty,\om}(\phi_{\infty,\om})
  =d_\infty(\om)
  =d_{\al}(\om),\\
  &K_{\al,\om}(\phi_{\infty,\om})
  =K_{\infty,\om}(\phi_{\infty,\om})
  =0. \end{align*}
This yields $\phi_{\infty,\om}\in\mc{M}_{\al,\om}$. From Lemma~\ref{lem:mg}, $\phi_{\infty,\om}$ is also a solution of
\eqref{eq:sp}. This means that $\phi_{\infty,\om}\in D(-\Del_\al)\cap H^2(\R^2)$. From the definition of the domain $D(-\Del_\al)$ and the singularity of $G_\om$, we see that $\phi_{\infty,\om}(0)=0$, which contradicts the positivity of $\phi_{\infty,\om}$.
\end{proof}

\begin{lemma}\label{lem:variconv}
Let $(v_n)_{n\in\N}$ be a sequence in $H_\al^1(\R^2)$ satisfy
\begin{align*}
  &K_{\om}(v_n)\to0,&
  &S_{\om}(v_n)\to d_\al(\om).
\end{align*}
Then there exist $v_0\in H_\al^1(\R^2)\setminus\{0\}$ and a subsequence $(v_{n_j})_{j\in\N}$ of $(v_n)_{n\in\N}$ such that $v_{n_j}\to v_0$ in $H_\al^1(\R^2)$ as $j\to\infty$.
In particular, $v_0\in\mc{M}_{\om}$.
\end{lemma}

\begin{remark}
If we just prove the existence of the minimizer $v_0\in\ml{M}_\om$, it suffices to only consider a minimizing sequence $(v_n)_{n\in\N}$ for $d_\al(\om)$, that is, $K_{\om}(v_n)=0$ and $S_{\om}(v_n)\to d_\al(\om)$. However, we show the stronger statement in Lemma~\ref{lem:variconv} because it is used when we apply the argument of Shatah~\cite{S83} for stability in Proposition~\ref{prop:scstabinsta}.
\end{remark}

\begin{proof}[Proof of Lemma~\ref{lem:variconv}]
We decompose $v_n=f_n+c_n G_\om$. From the assumptions and the expressions \eqref{eq:SHalK} and \eqref{eq:SLpK}, we have
\begin{align}\label{eq:compactconv}
  \frac{p-1}{2(p+1)}\|v_n\|_{H_{\al,\om}^1}^2
  &\to d_\al(\om), &
  \frac{p-1}{2(p+1)}\|v_n\|_{L^{p+1}}^{p+1}
  &\to d_\al(\om).
\end{align}
This implies that $(v_n)_n$ is bounded in $H_\al^1(\R^2)$, and so there exists $v_0\in H_\al^1(\R^2)$ and a subsequence of $(v_n)$ such that $v_n\wto v_0$ weakly in $H_\al^1(\R^2)$. From the definition of $H_\al^1(\R^2)$, we see that $f_n\wto f_0$ weakly in $H^1(\R^2)$ and $c_n\to c_0$ for some $(f_0,c_0)\in H^1(\R^2)\times\mb{C}$.

Now we show that $c_0\ne0$. Suppose that $c_0=0$. Then by \eqref{eq:compactconv} we have
\begin{align}\label{eq:c=0conv}
  \frac{p-1}{2(p+1)}(\|\nabla f_n\|_{L^2}^2+\om\|f_n\|_{L^2}^2)
  &\to d_\al(\om), &
  \frac{p-1}{2(p+1)}\|f_n\|_{L^{p+1}}^{p+1}
  &\to d_\al(\om).
\end{align}
Let
\begin{align*}
  \lam_n
  \ce\l(\frac{\|\nabla f_n\|_{L^2}^2+\om\|f_n\|_{L^2}^2}{\|f_n\|_{L^{p+1}}^{p+1}}\r)^{1/(p-1)}. \end{align*}
We have $K_{\infty,\om}(\lam_nf_n)=0$. Moreover, \eqref{eq:c=0conv} and Lemma~\ref{lem:posimu} imply $\lam_n\to1$. From the definition of $d_\infty(\om)$ and Lemma~\ref{lem:muineq}, we obtain
\begin{align*}
  d_\infty(\om)
  \le\frac{p-1}{2(p+1)}\|\lam_nf_n\|_{L^{p+1}}^{p+1}
  &\to d_\al(\om)<d_\infty(\om). \end{align*}
This is a contradiction, which implies $c_0\ne0$.

We show the strong convergence.
By the Brezis--Lieb Lemma \cite{BL83}, we have
\begin{align}
\label{eq:BLl}
    \|v_n\|_{H_{\al,\om}^1}^2
    -\|v_n-v_0\|_{H_{\al,\om}^1}^2
  &\to \|v_0\|_{H_{\al,\om}^1}^2, \\
\label{eq:BLk}
    K_{\om}(v_n)
    -K_{\om}(v_n-v_0)
  &\to K_{\om}(v_0). \end{align}
Since $\|v_0\|_{H_{\al,\om}^1}>0$, by \eqref{eq:BLl}, we have
\begin{align*}
  \lim_{n\to\infty}\frac{p-1}{2(p+1)}\|v_n-v_0\|_{H_{\al,\om}^1}^2
  <\lim_{n\to\infty}\frac{p-1}{2(p+1)}\|v_n\|_{H_{\al,\om}^1}^2
  =d_\al(\om).
\end{align*}
From this inequality and Lemma~\ref{lem:knega}, we have $K_{\om}(v_n-v_0)>0$ for large $n$. Therefore, from $K_{\om}(v_n)\to0$ and \eqref{eq:BLk}, we obtain $K_{\om}(v_0)\le0$. Thus, from \eqref{eq:variH1al2}, the lower semicontinuity of norms, we deduce that
\begin{align*}
  d_\al(\om)
  \le\frac{p-1}{2(p+1)}\|v_0\|_{H_{\al,\om}^1}^2
  \le\frac{p-1}{2(p+1)}\liminf_{n\to\infty}\|v_n\|_{H_{\al,\om}^1}^2
  =d_\al(\om). \end{align*}
From \eqref{eq:BLl}, we have $\|v_n-v_0\|_{H_{\al,\om}^1}^2\to0$.
Therefore, $v_n\to v_0$ in $H_\al^1(\R^2)$.
This completes the proof.
\end{proof}

\begin{lemma}\label{lem:gm}
$\mc{G}_{\om}\subset \mc{M}_{\om}$.
\end{lemma}

\begin{proof}
Let $\phi\in\mc{G}_{\om}$. Then $\phi\in\ml{K}_{\om}$. Since $\ml{M}_{\om}\ne\emptyset$ by Lemma~\ref{lem:variconv}, we can take $\psi\in\ml{M}_{\om}$. Moreover, by Lemma~\ref{lem:mg} we have $\psi\in\mc{G}_{\om}$. Therefore, for each $v\in\ml{K}_{\om}$ we obtain
\[S_{\om}(\phi)
  =S_{\om}(\psi)
  \le S_{\om}(v). \]
This implies $\phi\in\mc{M}_{\om}$. This completes the proof.
\end{proof}

\begin{proof}[Proof of Proposition~\ref{prop:eg}]
The assertion follows from Lemmas~\ref{lem:mg}, \ref{lem:variconv}, and \ref{lem:gm}.
\end{proof}

    \section{Symmetry of ground states}\label{sec:4}

In this section, we prove Theorem~\ref{thm:symmetryGS} based on the argument in \cite[Proof of Theorem~8.1.4]{C03} but need suitable modifications. We note that if $\phi\in \ml{A}_\om$, then we have
\[\|(-\Del_\al+\om)\phi\|_{L^2}
  \le \|\phi\|_{L^{2p}}^{p}
  \lesssim \|\phi\|_{H_{\al,\om}^1}^{p}, \]
which implies $\phi\in D(-\Del_\al)$. In particular, we can decompose $\phi=f+f(0)\beta_\al(\om)^{-1}G_\om$, and by \eqref{eq:opaction} and \eqref{eq:sp}, we have the relation
\begin{equation}\label{eq:spfphi}
  (-\Del+\om)f
  -|\phi|^{p-1}\phi
  =0.
\end{equation} Moreover, by the same argument in the proof of Lemma~\ref{lem:variconv}, we see that if $\phi=f+f(0)\beta_\al(\om)^{-1}G_\om\in\ml{G}_\om$, then $f(0)\ne0$.

\begin{lemma}\label{lem:scMom}
If $\psi\in H_{\al}^1(\R^2)$ satisfies
\begin{equation}\label{eq:HaldLp}
  \frac{p-1}{2(p+1)}\|\psi\|_{H_{\al,\om}^1}^2
  \le d_\al(\om)
  \le \frac{p-1}{2(p+1)}\|\psi\|_{L^{p+1}}^{p+1},
\end{equation}
then $\psi\in\ml{G}_\om$.
\end{lemma}

\begin{proof}
From the assumption \eqref{eq:HaldLp}, we have $K_{\om}(\psi)\le 0$ and $S_\om(\psi)\le d_\al(\om)$. On the other hand, by the first inequality in \eqref{eq:HaldLp} and Lemma~\ref{lem:knega}, we have $K_{\om}(\psi)\ge0$. Thus, $K_{\om}(\psi)=0$. Moreover, by the definition of $d_\al(\om)$, we obtain $d_\al(\om)\le S_\om(\psi)$. Therefore,  $\psi\in\ml{M}_\om=\ml{G}_\om$.
\end{proof}

Throughout this section, we denote the Schwartz symmetrization of $f$ by $f^*$.

\begin{lemma}\label{lem:3.1}
Let $\phi=f+f(0)\beta_\al(\om)^{-1}G_\om\in\mc{G}_{\om}$. Then
\[\|\phi\|_{L^{p+1}}^{p+1}
  \le\left\|f^*+\frac{|f(0)|}{\beta_\al(\om)}G_\om\right\|_{L^{p+1}}^{p+1}. \]
\end{lemma}

\begin{proof}
We note that
\begin{align*}
  \|\phi\|_{L^{p+1}}^{p+1}
  =\left\|f+\frac{f(0)}{\beta_\al(\om)}G_\om\right\|_{L^{p+1}}^{p+1}
  \le \left\||f|+\frac{|f(0)|}{\beta_\al(\om)}G_\om\right\|_{L^{p+1}}^{p+1}.
\end{align*}
After that, we only have to show that
\begin{equation}\label{eq:AL22}
  \left\||f|+\frac{|f(0)|}{\beta_\al(\om)}G_\om\right\|_{L^{p+1}}^{p+1}
  \le\left\|f^*+\frac{|f(0)|}{\beta_\al(\om)}G_\om\right\|_{L^{p+1}}^{p+1}.
\end{equation}

To prove \eqref{eq:AL22} we use \cite[Theorem~2.2]{AL89}. We denote
\[F(g,h)
  \ce (g+h)^{p+1} \]
for $g,h\in\R_+$. Then we have
\[F_{gh}(g,h)
  =p(p+1)(g+h)^{p-1}\ge 0 \]
for all $g,h\in\R_{\ge0}$. Therefore, from \cite[Theorem~2.2]{AL89}, we obtain
\begin{equation}\label{eq:FG}
  \int_{\R^2} F(g,h)\,dx\le \int_{\R^2} F(g^*,h^*)\,dx
\end{equation}
for all $g,h\in \ml{S}_0$. Here $\ml{S}_0$ is the set of all measurable functions $w\colon\R^2\to\R$ which satisfy
\begin{equation}\label{eq:asmpAL}
  w\ge 0,\quad
  \ml{L}^2(\{x\in\R^2\colon\, w(x)>a\}) < \infty\quad
  \text{for all $a > 0$},
\end{equation}
 where $\ml{L}^2(A)$ is the Lebesgue measure of the set $A$.

Moreover, we see that $|f|$ and $G_\om$ satisfy \eqref{eq:asmpAL} because $f\in H^2(\R^2)$ and $G_\om$ is decreasing. Thus, we can use the inequality \eqref{eq:FG} for $g=|f|$ and $h=|f(0)|\beta_\al(\om)^{-1}G_\om$. Since $G_\om^*=G_\om$, we obtain
\eqref{eq:AL22}. This completes the proof.
\end{proof}

\begin{lemma}\label{lem:3.2}
If $\phi=f+f(0)\beta_\al(\om)^{-1}G_\om\in\mc{G}_{\om}$, then $f^*+|f(0)|\beta_\al(\om)^{-1}G_\om\in\mc{G}_{\om}$.
\end{lemma}

\begin{proof}
Let $\psi\ce f^*+|f(0)|\beta_\al(\om)^{-1}G_\om$. We have
\begin{equation}\label{eq:psiL}
  \begin{aligned}
  \frac{p-1}{2(p+1)}\|\psi\|_{H_{\al,\om}^1}^2
  &=\frac{p-1}{2(p+1)}{\left(\|\nabla f^*\|_{L^2}^2
  +\om\| f^*\|_{L^2}^2
  +\frac{|f(0)|^2}{\beta_{\al}(\om)}\right)}
\\&\le\frac{p-1}{2(p+1)}{\left(\|\nabla f\|_{L^2}^2
  +\om\| f\|_{L^2}^2
  +\frac{|f(0)|^2}{\beta_{\al}(\om)}\right)}
\\&=\frac{p-1}{2(p+1)}\|\phi\|_{H_{\al,\om}^1}^2
  =d_\al(\om).
  \end{aligned}
\end{equation}
Moreover, from Lemma~\ref{lem:3.1}, we have
\[\frac{p-1}{2(p+1)}\|\psi\|_{L^{p+1}}^{p+1}
  \ge\frac{p-1}{2(p+1)}\|\phi\|_{L^{p+1}}^{p+1}
  =d_\al(\om). \]
Therefore, Lemma~\ref{lem:scMom} implies $\psi\in\mc{G}_{\om}$.
\end{proof}

\begin{lemma}\label{lem:posif}
If $\phi=f+f(0)\beta_\al(\om)^{-1}G_{\om}\in\mc{G}_{\om}$ and $f(0)>0$, then $f$ is a positive function.
\end{lemma}

\begin{proof}
Let
\begin{align*}
  &g\ce\lvert\Re f\rvert,&
  &h\ce\lvert\Im f\rvert,&
  &\psi\ce g+ih+\frac{f(0)}{\beta_\al(\om)}G_\om.
\end{align*}
Then we see that
\begin{align*}
  &\|\psi\|_{H_{\al,\om}^1}
  =\|\phi\|_{H_{\al,\om}^1},&
  &\|\psi\|_{L^{p+1}}
  \ge\|\phi\|_{L^{p+1}}.
\end{align*}
Therefore, Lemma~\ref{lem:scMom} implies $\psi\in\mc{G}_{\om}$. Since
\[\psi
  =g+ih
  +\frac{(g+ih)(0)}{\beta_\al(\om)}G_\om
  \in D(-\Del_\al), \]
we can use \eqref{eq:opaction} and obtain
\[(-\Del+\om)(g+ih)
  =(-\Del_\al+\om)\psi
  =|\psi|^{p-1}\psi
  =|\psi|^{p-1}\biggl(g+ih+\frac{f(0)}{\beta_\al(\om)}G_\om\biggr). \]
We decompose it into the real part and the imaginary part as
\begin{align}\label{eq:veq}
  (-\Del+\om)g
  &=|\psi|^{p-1}\biggl(g+\frac{f(0)}{\beta_\al(\om)}G_\om\biggr)
\\\label{eq:weq}
  (-\Del+\om)h
  &=|\psi|^{p-1}h.
\end{align}
Since each right-hand side of \eqref{eq:veq} and \eqref{eq:weq} is in $L^2(\R^2)$, we see that $g,h\in H^2(\R^2)\subset H^1(\R^2)\cap C(\R^2)$. Therefore, since $g(0)=f(0)>0$ and $g,G_\om\ge0$ in $\R^2$, applying the strong maximal principle (e.g., \cite[Theorem~3.1.2]{C18}) to the solution $g$ of \eqref{eq:veq}, we have $g>0$ in $\R^2$. Similarly, by using $h(0)=0$ and \eqref{eq:weq}, we obtain $h\equiv 0$ in $\R^2$.

From the continuity of $f$ and the positivity of $g$, we see that the sign of $f(x)$ does not depend on $x$. Therefore, by $f(0)>0$ we have $f>0$ in $\R^2$. This completes the proof.
\end{proof}

\begin{lemma}\label{lem:raddecf}
If $\phi=f+f(0)\beta_\al(\om)^{-1}G_{\om}\in\mc{G}_{\om}$ and $f$ is positive, then $f$ is a radial and strictly decreasing function.
\end{lemma}

\begin{proof}
We denote
\[\psi
  \ce f^*+\frac{f(0)}{\beta_\al(\om)}G_\om. \]
From Lemma \ref{lem:3.2} we have $\psi\in\mc{G}_{\om}$. In particular, $\psi\in D(-\Del_\al)$ and so $f^*(0)=f(0)$. Moreover, since  $\|\psi\|_{H_{\al,\om}^1}^2=\|\phi\|_{H_{\al,\om}^1}^2$ by $\phi,\psi\in\ml{M}_\om$, we see that
\begin{equation}\label{eq:BZasmp1}
  \|\nabla f^*\|_{L^2}
  =\|\nabla f\|_{L^2}.
\end{equation}

Now we show that $f^*$ is strictly decreasing. Since $\psi=f^*+f^*(0)\beta_\al(\om)^{-1}G_\om$ is a positive radial solution of \eqref{eq:spfphi}, $f^*(r)$ satisfies
\begin{equation}\label{eq:f*ode}
  -f''
  -\frac{1}{r}f'
  +\om f
  =\left(f+\frac{f(0)}{\beta_\al(\om)}G_\om\right)^p,\quad
  r>0.
\end{equation}
Suppose that $f^*$ is not strictly decreasing. Then $f^*$ is constant in some interval $I=(r_1,r_2)$. From the equation \eqref{eq:f*ode}, $f^*$ satisfies
\[\om f^*
  =\left(f^*+\frac{f^*(0)}{\beta_\al(\om)}G_\om\right)^p\quad
  \text{in $(r_1,r_2)$}. \]
The left hand side is a constant whereas the right hand side is not a constant on $(r_1,r_2)$ since $G_\om$ is a strictly decreasing function and $f^*(0)=f(0)>0$. This is a contradiction. Thus, $f^*$ is strict decreasing.

Since $f^*(r)$ is  strictly decreasing, we see that the Lebesgue measure of the set
\begin{equation} \label{eq:BZasmp2}
  \{x\in\R^2\colon\,
  \nabla f^*(x)=0,\ 0<f^*(|x|)<\|f^*\|_{L^\infty}\}
\end{equation}
is zero. Therefore, combining \eqref{eq:BZasmp1}, we can use \cite[Theorem~1.1]{BZ88} to see that there exists $y\in\R^2$ such that $f(\cdot-y)=f^*$. Moreover, from $f(0)=f^*(0)$ we obtain $f=f^*$. This completes the proof.
\end{proof}

\begin{proof}[Proof of Theorem~\ref{thm:symmetryGS}]
Let $\phi=f+f(0)\beta_\al(\om)^{-1}G_\om\in\ml{G}_\om$. Since $f(0)\ne0$, there exists $\theta\in\R$ such that $e^{i\theta}f(0)>0$. By Lemmas \ref{lem:posif} and \ref{lem:raddecf} we see that $e^{i\theta}f$ is a positive, radial, and decreasing function. This completes the proof.
\end{proof}

    \section{Rescaled limit}\label{sec:5}

In this section, we prove that a family of rescaled positive ground states converges to the positive radial ground state of \eqref{eq:nonlineq}. The argument is based on \cite{FO03I}.

Let $\phi=f+f(0)\beta_\al(\om)^{-1}G_\om\in\mc{A}_{\om}$ and define the rescaling
\begin{align*}
  &\til\phi (x) \ce \om^{-1/(p-1)} \phi ( x / \sqrt{\om}),
\\&\til\al = \til\al (\om) \ce \al + \frac{1}{4\pi} \ln \om.
\end{align*}
Since
\[\begin{aligned}
  \beta_\al(\om)
  &=\al
  +\frac{\gamma}{2\pi}
  +\frac{1}{2\pi}\ln\frac{\sqrt{\om}}{2}
\\&=\al
  +\frac{1}{4\pi}\ln\om
  +\frac{\gamma}{2\pi}
  +\frac{1}{2\pi}\ln\frac{1}{2}
  =\beta_{\til\al}(1),
\\G_\om(\lam x)
  &=G_{\lam^2\om}(x)\quad
  \text{for $\lam>0$},
\end{aligned} \]
we have the decomposition
\begin{equation}\label{eq:expresstilphi}
  \til\phi
  =\til f
  +\frac{\til f(0)}{\beta_{\til\al}(1)}G_1
  \in D(-\Del_{\til\al}).
\end{equation}
From this expression, \eqref{eq:opaction}, and $f(x)=\om^{1/(p-1)}\til f(\sqrt{\om}\, x)$, we have the relations
\begin{align*}
  (-\Del_\al+\om)\phi
  &=(-\Del+\om)f
\\&=\om^{1/(p-1)}\om(-\Del+1)\til f(\sqrt{\om}\mathrel{\cdot})
\\&=\om^{p/(p-1)}(-\Del_{\til\al}+1)\til \phi(\sqrt{\om}\mathrel{\cdot}),
\\\om\phi
  &=\om^{p/(p-1)}\til{\phi}(\sqrt{\om}\mathrel{\cdot}),
\\|\phi|^{p-1}\phi
  &=\om^{p/(p-1)}|\til{\phi}(\sqrt{\om}\mathrel{\cdot})|^{p-1}\til{\phi}(\sqrt{\om}\mathrel{\cdot}).
\end{align*}
Therefore, $\til\phi$ is a solution of
\begin{equation}\label{eq:tilsp}\mathopen{}
  -\Del_{\til{\al}} \til\phi + \til\phi - |\til\phi|^{p-1} \til\phi = 0.
\end{equation}
Noting that $\beta_{\til\al}(1)=\beta_\al(\om)$, we denote
\[\|v\|_{\til H_{\al,\om}^1}
  \ce\|v\|_{H_{\til\al(\om),1}^1}
  =\sqrt{\|f\|_{H^1}^2
  +\beta_{\al}(\om)|c|^2}\quad
  \text{for } v=f+cG_1\in H_\al^1(\R^2). \]
The action and the Nehari functional corresponding to \eqref{eq:tilsp} are given by
\begin{align*}
  &\til S_{\al,\om}(v)
  =\frac12\|v\|_{\til H_{\al,\om}^1}^2
  -\frac{1}{p+1}\|v\|_{L^{p+1}}^{p+1},&
  &\til K_{\al,\om}(v)
  =\|v\|_{\til H_{\al,\om}^1}^2
  -\|v\|_{L^{p+1}}^{p+1}
\end{align*}
for $v\in H_\al^1(\R^2)$, respectively. The action and the Nehari functional corresponding to the limit equation \eqref{eq:tilsp} are given by
\begin{align*}
  &\til S_{\infty}(v)
  =\frac12\|v\|_{H^1}^2
  -\frac{1}{p+1}\|v\|_{L^{p+1}}^{p+1},&
  &\til K_{\infty}(v)
  =\|v\|_{H^1}^2
  -\|v\|_{L^{p+1}}^{p+1}
\end{align*}
for $v\in H^1(\R^2)$, respectively. Note that since
\[e_{\til\al}
  =e_{\al}\om^{-1}, \]
we have
\[\om>-e_{\al}
  \iff 1>-e_{\til\al(\om)}. \]

In what follow, we only consider the ground state with the positive regular part:
\begin{equation}\label{eq:posigs}
  \phi_{\om}
  =f_\om+\frac{f_\om(0)}{\beta_{\al}(\om)}G_\om\in\mc{G}_{\om},\quad
  f_\om>0,\quad
  f_\om\in H_{\rad}^2(\R^2).
\end{equation}
Note that from Theorem~\ref{thm:symmetryGS}, for any ground state $\psi_\om\in\ml{G}_{\om}$ there exists $\theta\in\R$ such that $\phi_{\om}\ce e^{i\theta}\psi_\om$ satisfies \eqref{eq:posigs}.

\begin{remark}
If $\phi$ is decompose as $\phi=f+f(0)\beta_\al(\om)^{-1}G_\om$, then it is natural to decompose $\til\phi$ as \eqref{eq:expresstilphi}. However, since $\beta_{\til\al} (1)=\beta_\al(\om)$, for simplicity of notations, we always decompose it as
\[\til\phi
  =\til f+\frac{\til f(0)}{\beta_\al(\om)}G_1. \]
\end{remark}

We denote
\[\begin{aligned}
  \til d_\al(\om)
  &\ce\til S_{\al,\om}(\til{\phi}_{\om}),
\\\til d(\infty)
  &\ce \til S_{\infty}(\til{\phi}_{\infty})
  =\frac{p-1}{2(p+1)}\|\til\phi_\infty\|_{L^{p+1}}^{p+1},
  \end{aligned} \]
where $\til{\phi}_{\infty}$ is the unique positive radial solution of \eqref{eq:nonlineq}.

\begin{proposition}\label{lem:4.1}
Let $(\phi_{\om})_{\om>-e_\al}$ be a family of ground states for \eqref{eq:sp} satisfying \eqref{eq:posigs}. Then
\begin{equation}\label{eq:convom}
  \til{f}_{\om}
  \to \til{\phi}_{\infty}\quad
  \text{in $H^2(\R^2)$ as $\om\to\infty$}.
\end{equation}
In particular, $\til{\phi}_{\om}\to\til{\phi}_{\infty}$ in $H_\al^1(\R^2)$ as $\om\to\infty$.
\end{proposition}

\begin{proof}
We divide the proof into several steps.


\begin{enumerate}[nosep, leftmargin=0pt, itemindent=*, parsep=0pt, label=\textbf{Step~\arabic*.}]
\item $\om<\om'<\infty\implies\til d_\al(\om)<\til d_\al(\om')\le\til d(\infty)$.

Since
\[\til K_{\om}(\til\phi_{\om'})
  <\til K_{\om'}(\til\phi_{\om'})
  =0, \]
by Lemma~\ref{lem:knega} we have
\[\til d_\al(\om)
  <\frac{p-1}{2(p+1)}\|\til\phi_{\om'}\|_{L^{p+1}}^{p+1}
  =\til d(\om'). \]

Similarly, since
\[\til K_{\om}(\til\phi_{\infty})
  =\til K_{\infty}(\til\phi_{\infty})
  =0,  \]
we obtain
\[\til d_\al(\om)
  \le\frac{p-1}{2(p+1)}\|\til \phi_{\infty}\|_{L^{p+1}}^{p+1}
  =\til d(\infty). \]

\item $\sup_{\om>-e_\al}\|\til\phi_{\om}\|_{\til H_{\al,\om}^1}<\infty$ and $\inf_{\om>1-e_\al}\|\til\phi_{\om}\|_{L^{p+1}}>0$.

This follows from the expression
\[\til d_\al(\om)
  =\frac{p-1}{2(p+1)}\|\til\phi_{\om}\|_{\til H_{\om}^1}^2
  =\frac{p-1}{2(p+1)}\|\til\phi_{\om}\|_{L^{p+1}}^{p+1} \]
and Step~1.

\item Weak convergence to a positive radial function.

If we decompose $\til{\phi}_{\om}=\til{f}_{\om}+\til{f}_{\om}(0)\beta_{\al}(\om)^{-1}G_1$, then by Step 2, there exist nonnegative radial function $f_\infty\in H^1(\R^2)$, a constant $c_\infty\in\R$, and a subsequence $(\om_j)_{j\in\N}$ such that $\om_j\to\infty$ and
\begin{align}
  \label{eq:tilfomweakconv}
  &\til f_{\om_j}
  \wto \til f_\infty\quad
  \text{weakly in }H^1(\R^2),
\\\label{eq:tilfomcinfconv}
  &\frac{\til f_{\om_j}(0)^2}{\beta_{\al}(\om_j)}
  \to c_\infty
\end{align}
as $j\to\infty$. By \eqref{eq:tilfomweakconv} and the radial compactness lemma, we see that
\[\til f_{\om_j}
  \to\til f_\infty\quad
  \text{in $L^{q}(\R^2)$ as $j\to\infty$ for any $q > 2$}. \]
Since $\beta_{\al}(\om_j)\to\infty$, by \eqref{eq:tilfomcinfconv} we have
\[\frac{\til f_{\om_j}(0)}{\beta_{\al}(\om_j)}
  \to 0\quad
  \text{as $j\to\infty$}. \]
Therefore, we obtain
\[\til\phi_{\om_{j}}
  \to\til f_\infty\quad
  \text{in $L^{q}(\R^2)$ as $j\to\infty$ for any $q > 2$}. \]
Moreover, from Step~2 again, we see that $\til f_\infty\ne0$.

\item
$\til f_\infty=\til{\phi}_\infty$.

From the equation
\[-\Del\til f_\om
  +\til f_\om
  -\til{\phi}_{\om}^p
  =0 \]
and Step~3, we see that $\til f_\infty$ is a weak solution of Equation~\eqref{eq:nonlineq}. Therefore, the uniqueness of the nonnegative, radial, and decreasing solutions of \eqref{eq:nonlineq}, we obtain $\til f_\infty=\til{\phi}_\infty$.

\item
Strong convergence in $H^2(\R^2)$, i.e.,
\begin{equation}\label{eq:H2conv}
  \til f_{\om_{j}}
  \to \til{\phi}_\infty\quad
  \text{in }H^2(\R^2).
\end{equation}

From the equations we have
\begin{align*}
  |(-\Del+1)(\til{f}_{\om_{j}}-\til{\phi}_\infty)|
  =|\til{\phi}_{\om_{j}}^p-\til{\phi}_\infty^p|
  \lesssim(\til{\phi}_{\om_{j}}^{p-1}+\til{\phi}_{\infty}^{p-1})|\til{\phi}_{\om_{j}}-\til{\phi}_\infty|.
\end{align*}
Thus, by Steps 3 and 4, we obtain
\begin{align*}
  \|\til{f}_{\om_{j}}-\til{\phi}_\infty\|_{H^2}^2
  &\lesssim \int(\til{\phi}_{\om_{j}}^{2(p-1)}+\til{\phi}_{\infty}^{2(p-1)})|\til{\phi}_{\om_{j}}-\til{\phi}_\infty|^2\,dx
\\&\lesssim (\|\til{\phi}_{\om_{j}}\|_{L^{2(p+1)}}^{2(p-1)}
  +\|\til{\phi}_{\infty}\|_{L^{2(p+1)}}^{2(p-1)})
  \|\til{\phi}_{\om_{j}}-\til{\phi}_\infty\|_{L^{p+1}}^2
  \to0.
\end{align*}
This means that \eqref{eq:H2conv} holds.

\item Conclusion.

The above argument works if we start with any subsequence of $(\til\phi_{\om})_{\om>-e_\al}$. Thus we have
\[\til f_\om
  \to \til\phi_\infty \quad
  \text{in $H^2(\R^2)$ as $\om\to\infty$. } \]
\end{enumerate}
This completes the proof.
\end{proof}

    \section{Nondegeneracy and lower boundedness of linearized operator in radial function space}\label{sec:6}

In this section, we prove lower boundedness and nondegeneracy of the linearlized operator around rescaled ground states with the large frequency in the radial function space. We follow the argument in \cite{Fs,F05}.

We denote
\begin{align*}
  D_\rad(-\Del_\al;\R)
  &\ce \left\{f+\frac{f(0)}{\beta_\al(\lam)}G_\lam\colon\,
  f\in H_\rad^2(\R^2;\R)\right\},
\\H_{\al,\rad}^1(\R^2;\R)
  &\ce \{f+cG_\lam\colon\, f\in H_\rad^1(\R^2;\R),\ c\in\mb{R}\}.
\end{align*}

Let $(\phi_\om)_{\om>-e_\al}$ be a family of positive ground states. We define the linearized operator around $\til\phi_\om$ by
\begin{align*}
  \til{L}_\om v
  =\til{L}_{\al,\om} v
  &\ce(-\Del_{\til{\al}(\om)}+1)v
  -p\til{\phi}_{\om}^{p-1}v \quad
  \text{for } v\in H_\al^1(\R^2;\R).
\end{align*}
We denote
\[\til{L}_{\infty} f
  \ce(-\Del+1)f
  -p\til{\phi}_{\infty}^{p-1}f\quad
  \text{for $f\in H^2(\R^2;\R)$}. \]
It is known (see, e.g., \cite{Fs}) that there exists $C>0$ such that
\begin{align}
  \label{eq:Linfbddness}
  \|f\|_{H^2}
  &\le C\|\til{L}_{\infty} f\|_{L^2} \quad
  \text{for all $f\in H_\rad^2(\R^2;\R)$},
\\\label{eq:Linfbddness2}
  \|f\|_{H^1}
  &\le C\|\til{L}_{\infty} f\|_{H_{\rad}^{-1}} \quad
  \text{for all $f\in H_\rad^1(\R^2;\R)$}.
\end{align}

\begin{lemma}\label{lem:ftL}
There exist $\om_1>-e_\al$ and $C>0$ such that for all $\om>\om_1$,
\begin{equation}
  \|v\|_{\til{H}_{\al,\om}^1}
  \le C\|\til{L}_{\om} v\|_{\til{H}_{\al,\om,\rad}^{-1}}\quad
  \text{for all }
  v=f+cG_1\in H_{\al,\rad}^1(\R^2;\R).
\end{equation}
In particular, if $v\in H_{\al,\mr{rad}}^1(\R^2;\R)$ satisfies $\til{L}_{\om} v=0$, then $v=0$.
\end{lemma}

\begin{remark}\label{rem:6.2}
Lemma~\ref{lem:ftL} means that zero is not an eigenvalue of the operator $\til{L}_{\om}\colon D_\rad(-\Del_{\til\al};\R)\to L_\rad^2(\R^2;\R)$. In fact, by essential spectral theorem, we see that zero is its resolvent.
\end{remark}

\begin{proof}[Proof of Lemma~\ref{lem:ftL}]
For $v=f+cG_1$, $w=g+dG_1\in H_{\al,\rad}^1(\R^2;\R)$, we have the expression
\begin{align*}
  \langle\til{L}_{\om} v,w\rangle
  ={}&\langle\til{L}_{\om} (f+cG_1),g+dG_1\rangle
\\={}&(\nabla f,\nabla g)_{L^2}
  +(f,g)_{L^2}
  +cd\beta_\al(\om)
  -p(\til{\phi}_{\om}^{p-1}v,w)_{L^2}
\\={}&\langle\til{L}_\infty f,g\rangle
  +cd\beta_\al(\om)
  +p(\til{\phi}_{\infty}^{p-1}f,g)_{L^2}
  -p(\til{\phi}_{\om}^{p-1}v,w)_{L^2}
\\={}&\langle\til{L}_\infty f,g\rangle
  +cd\beta_\al(\om)
  +p\int(\til{\phi}_{\infty}^{p-1}-\til{\phi}_{\om}^{p-1})fg
\\&-p\int\til{\phi}_{\om}^{p-1}(df+cg)G_1
  -pcd\int\til{\phi}_{\om}^{p-1}G_1^2.
\end{align*}
Therefore, using the estimate \eqref{eq:Linfbddness2}, we have
\begin{align*}
   \| \til{L}_{\om} v \|_{\til{H}_{\al,\om,\rad}^{-1}}
   &= \sup\{\langle \til{L}_{\om} v, w \rangle \colon\, \|w\|_{\til{H}_{\al,\om,\rad}^1} \le 1\}
\\ &\ge \sup\biggl\{\langle \til{L}_{\om} (f+cG_1), g+dG_1 \rangle\colon\, \substack{\|g\|_{H_\rad^1}^2\le 1/2, \\ d^2\beta_\al(\om)\le1/2} \biggr\}
\\ &= \sup\biggl\{\begin{aligned}[t]
   &\langle \til{L}_\infty f , g \rangle
   + c d \beta_\al(\om)
   + p \int(\til{\phi}_{\infty}^{p-1}-\til{\phi}_{\om}^{p-1}) f g
\\ &- p\int\til{\phi}_{\om}^{p-1}(df+cg)G_1
   -pcd\int\til{\phi}_{\om}^{p-1}G_1^2 \colon\,
   \substack{\|g\|_{H_\rad^1}\le 2^{-1/2}, \\ |d|\le (2\beta_\al(\om))^{-1/2}}
   \biggr\}
   \end{aligned}
\\ &\gtrsim\|\til{L}_\infty f\|_{H_\rad^{-1}}
  +\beta_\al(\om)^{1/2}|c|
\\ &-\|\til{\phi}_{\infty}-\til{\phi}_{\om}\|_{L^{p+1}}\|f\|_{L^{p+1}}
   -\frac{\|f\|_{H^1}}{\beta_\al(\om)^{1/2}}
   -|c|
   -\frac{|c|}{\beta_\al(\om)^{1/2}}
\\ &\gtrsim\|f\|_{H^{1}}
   +\beta_\al(\om)^{1/2}|c|
   \simeq\|v\|_{\til{H}_{\al,\om}^1}
\end{align*}
for large $\om$, where we used the fact $\beta_\al(\om)\to\infty$ as $\om\to\infty$. This completes the proof.
\end{proof}

    \section{Uniqueness of ground states for large frequencies}\label{sec:7}

In this section, we prove the uniqueness of ground states for large frequencies (Theorem~\ref{thm:uniqG}). The proof is based on \cite{F05}.

\begin{lemma}\label{lem:uniqueGi}
There exist $\delta>0$ and $\om_1>-e_\al$ such that the following holds. If $(\phi_\om)_{\om>-e_\al}$ is a family of positive ground states with $\phi_\om\in\ml{G}_\om$, $\om_0>\om_1$, and $\psi\in\ml{A}_{\om_0}$ is a real-valued radial solution satisfying $\|\til{\psi}-\til{\phi}_\infty\|_{H_\al^1}<\delta$, then $\psi=\phi_{\om_0}$.
\end{lemma}

\begin{proof}
We take a small $\delta>0$ to be chosen later, and let $\|\til{\psi}-\til{\phi}_\infty\|_{H_\al^1}<\delta$. We define the operator
\[\til{L}_{\om}^* v
  \ce (-\Del_{\til{\al}(\om)}+1)v
  -V_\om(x)v, \]
where
\[V_\om(x)
  \ce
  \left\{\begin{array}{@{}c l @{}}
  \dfrac{\til{\phi}_{\om}(x)^p-|\til{\psi}(x)|^{p-1}\til{\psi}(x)}{\til{\phi}_{\om}(x)-\til{\psi}(x)}
  &\text{if }\phi_{\om}(x)\ne\psi(x),
\\[4\jot]
  p\til{\phi}_{\om}(x)^{p-1}
  &\text{if }\phi_{\om}(x)=\psi(x).
  \end{array}\right. \]

First, we show that there exist $\om_1>-e_\al$ and $C>0$ such that if $\om>\om_1$, then
\begin{equation}\label{eq:ftLs}
  \|f\|_{H^2}
  \le C\|\til{L}_{\om}^* v\|_{L^2}\quad
  \text{for all }
  v=f+\frac{f(0)}{\beta_{\al}(\om)}G_1\in D_{\rad}(-\Del_{\til\al};\R).
\end{equation}
To prove this, we will show that
\begin{equation}\label{eq:convV}
  \|V_\om-p\til{\phi}_\infty^{p-1}\|_{L^q}
  \lesssim \delta\quad
  \text{for all $q\ge 2$ and large $\om$}.
\end{equation}
We can rewrite
\[V_\om(x)
  =p\int_0^1|\til{\phi}_{\om}(x)+sr_\om(x)|^{p-1}\,ds,\quad
  r_\om(x)
  \ce\til\psi(x)-\til\phi_{\om}(x). \]
Then we have
\begin{align*}
  &|V_\om-p\til{\phi}_\infty^{p-1}|
\\&\le p\int_0^1\left||\til{\phi}_{\om}+sr_\om|^{p-1}-\til{\phi}_\infty^{p-1}\right|\,ds
\\&\lesssim\left\{\begin{aligned}
  &(\til{\phi}_{\om}^{p-2}+|r_\om|^{p-2}+\til{\phi}_\infty^{p-2})(|\til{\phi}_{\om}-\til{\phi}_\infty|+|r_\om|)
  &&\text{if $p>2$}
\\&|\til{\phi}_{\om}-\til{\phi}_\infty|
  +|r_\om|
  &&\text{if $1<p\le 2$}.
  \end{aligned}\right.
\end{align*}
We note that Proposition~\ref{lem:4.1} implies $\|\til{\phi}_\om - \til{\phi}_\infty\|_{H_\al^1} < \delta$ for large $\om$. Therefore, if $p > 2$, then
\begin{align*}
  \| V_\om - p \til{\phi}_\infty^{p-1} \|_{L^q}
  \lesssim{}&( \| \til{\phi}_{\om} \|_{ L^{(p-1)q} }^{p-2}
  +\| r_{\om} \|_{ L^{(p-1)q} }^{p-2}
  +\| \til{\phi}_{\infty} \|_{ L^{(p-1)q} }^{p-2} )
\\&\cdot(\| \til{\phi}_{\om} - \til{\phi}_\infty \|_{ L^{(p-1)q} }
  +\|r_\om\|_{L^{(p-1)q}})
\\\lesssim{}&\| \til{\phi}_{\om} - \til{\phi}_\infty \|_{H_\al^1}
  +\| \til{\psi} - \til{\phi}_\infty \|_{ H_\al^1 }
  \lesssim \delta.
\end{align*}
If $1<p\le 2$, we have
\begin{align*}
  \|V_\om-p\til{\phi}_\infty^{p-1}\|_{L^q}
  &\lesssim \|\til{\phi}_{\om}-\til{\phi}_\infty\|_{L^{q}}
  +\|r_\om\|_{L^q}
\\&\lesssim \|\til{\phi}_{\om}-\til{\phi}_\infty\|_{H_\al^1}
  +\|\til{\psi}-\til{\phi}_\infty\|_{H_\al^1}
  \lesssim \delta.
\end{align*}
Therefore, we have \eqref{eq:convV}.

Next, we show the estimate \eqref{eq:ftLs}. By using the expression
\begin{align*}
  \til{L}_{\om}^* v
  &=\til{L}_{\infty} f
  +p\til{\phi}_{\infty}^{p-1}f
  -V_\om(x) v
\\&=\til{L}_{\infty} f
  +(p\til{\phi}_{\infty}^{p-1}-V_\om(x))f
  -\frac{f(0)}{\beta_{\al}(\om)}V_\om(x)G_1
\end{align*}
for $v=f+f(0)\beta_\al(\om)^{-1}G_1\in D_{\rad}(-\Del_{\til\al};\R)$,  \eqref{eq:Linfbddness}, \eqref{eq:convV}, and the embedding $|f(0)|\lesssim\|f\|_{H^2}$ we obtain
\begin{align*}
  \|\til{L}_{\om}^*v\|_{L^2}
  &\ge \|\til{L}_{\infty} f\|_{L^2}
  -\|(p\til{\phi}_{\infty}^{p-1}-V_\om)f\|_{L^2}
  -\frac{|f(0)|}{\beta_{\al}(\om)}\|V_\om G_1\|_{L^2}
\\&\gtrsim \|f\|_{H^2}-\delta\|f\|_{H^2}
  \simeq\|f\|_{H^2}
\end{align*}
if $\delta$ is small and $\om$ is large sufficiently. This implies that the estimate \eqref{eq:ftLs} holds.

Finally, from the equation \eqref{eq:tilsp}, we have $\til{L}_{\om}^* (\til{\phi}_{\om} - \til{\psi}) = 0$. By \eqref{eq:ftLs} we obtain $\til{\phi}_{\om}=\til{\psi}$. This completes the proof.
\end{proof}

\begin{proof}[Proof of Theorem~\ref{thm:uniqG}]
The assertion follows from Theorem~\ref{thm:symmetryGS}, Proposition~\ref{lem:4.1}, and Lemma~\ref{lem:uniqueGi}.
\end{proof}

    \section{Regularity of $\om\mapsto\phi_{\om}$}\label{sec:8}

In this section, we verify the differentiability of $\om\mapsto\phi_\om$ for large $\om$ following the argument in \cite[Section~6]{SS85} with modifications.

\begin{proposition}\label{prop:reg1}
Let $\om_1>-e_\al$ be as in Lemma~\ref{lem:ftL} and Lemma~\ref{lem:uniqueGi} and let $\phi_{\om}=f_\om+f_\om(0)\beta_\al(\om)^{-1}G_\om\in\ml{G}_{\om}$ be the unique positive radial ground state given in Theorem~\ref{thm:uniqG}. Then the map $\om\mapsto\til{f}_\om$ is in $C^1((\om_1,\infty),H_\rad^2(\R^2;\R))$.
\end{proposition}

\begin{proof}
We define the function $F$ by
\begin{align*}
  F(\om,f)
  \ce f-(1-\Del)^{-1}\left[\left|f+\frac{f(0)}{\beta_\al(\om)}G_1\right|^{p-1}\left(f+\frac{f(0)}{\beta_\al(\om)}G_1\right)\right],
\\(\om,f)\in (-e_\al,\infty)\times H_{\rad}^2(\R^2;\R).
\end{align*}
Then we have
\begin{align*}
  \frac{\pt F}{\pt\om}(\om,f)
  ={}&-\frac{pf(0)}{4\pi\om\beta_\al(\om)^2}(1-\Del)^{-1}\left[\left|f+\frac{f(0)}{\beta_\al(\om)}G_1\right|^{p-1}G_1\right],
\end{align*}
and for $w\in H_{\rad}^2(\R^2;\R)$,
\begin{align*}
  \frac{\pt F}{\pt f}(\om,f)w
  ={}&w-p(1-\Del)^{-1}\left[\left|f+\frac{f(0)}{\beta_\al(\om)}G_1\right|^{p-1}\left(w+\frac{w(0)}{\beta_\al(\om)}G_1\right)\right].
\end{align*}
From this expression we have
\[F\in C^1\bigl((-e_\al,\infty)\times H_{\rad}^2(\R^2;\R),H_{\rad}^2(\R^2;\R)\bigr). \]

Let $\om_0>\om_1$. We have
\[F(\om_0,\til{f}_{\om_0})
  =(1-\Del)^{-1}\til{S}_{\om_0}'(\til{\phi}_{\om_0})
  =0. \]
Moreover, since the operator $\til{L}_{\om_0}=(1-\Del_{\til{\al}})-p\til{\phi}_{\om_0}^{p-1}\colon D_\rad(-\Del_{\til{\al}};\R)\to L^2(\R^2;\R)$ is invertible by Remark~\ref{rem:6.2} and since the map $\tau_\om\colon H_{\rad}^2(\R^2;\R)\to D_\rad(-\Del_{\til{\al}};\R)$; $w\mapsto w+w(0)\beta_{\til{\al}}(1)^{-1}G_1$ is also invertible by the definition of $D_\rad(-\Del_{\til{\al}};\R)$, we see that the operator
\begin{align*}
  \frac{\pt F}{\pt f}(\om_0,\til{f}_{\om_0})
  &=I-(1-\Del)^{-1}p\til{\phi}_{\om_0}^{p-1}
\\&=(1-\Del)^{-1}\til{L}_{\om_0}\tau_{\om_0}
  \colon H_{\rad}^2(\R^2;\R)\to H_{\rad}^2(\R^2;\R)
\end{align*}
is also invertible. Therefore, by the implicit function theorem, there exists a $C^1$-curve $\om\mapsto g_\om$ defined on a neighborhood of $\om_0$ into $H_{\rad}^2(\R^2;\R)$ such that $F(\om,g_\om)=0$ and $g_{\om_0}=\til{f}_{\om_0}$. From Lemma~\ref{lem:uniqueGi} we have $g_{\om}=\til{f}_\om$ for $\om$ around $\om_0$. This completes the proof.
\end{proof}

Note that by a standard elliptic regularity argument, one can obtain the spatially exponential decay of $f_\om$. Thus, by Proposition~\ref{prop:reg1} and the definition of the rescaling $\til f_{\om}$, we obtain the regularity of $\om\mapsto f_\om$. In particular, we have

\begin{corollary}\label{cor:reg}
Let $\om_1>-e_\al$ be as in Lemma~\ref{lem:ftL} and Lemma~\ref{lem:uniqueGi} and let $\phi_{\om}$ be the unique positive radial ground state given in Theorem~\ref{thm:uniqG}. Then the map $\om\mapsto\phi_\om$ is in $C^1((\om_1,\infty),H_{\al,\rad}^1(\R^2;\R))$.
\end{corollary}

    \section{Stability and instability for large frequencies}\label{sec:9}

In this section, we calculate the value $\frac{d}{d\om}\|\phi_\om\|_{L^2}^2$ for large $\om$ based on \cite{Fs,F05}. We prove the following.

\begin{proposition}\label{prop:stabinsta}
Let $(\phi_\om)_{\om>\om_1}$ be the family of the unique positive ground states obtained in Theorem~\ref{thm:uniqG}. Then there exists $\om^*=\om^*(p)\in(\om_1,\infty)$ such that the following is true.
\begin{itemize}
\item If $1<p\le 3$, then  $\frac{d}{d\om}\|\phi_\om\|_{L^2}^2>0$ for all $\om>\om^*$.
\item If $p>3$, then  $\frac{d}{d\om}\|\phi_\om\|_{L^2}^2<0$ for all $\om>\om^*$.
\end{itemize}
\end{proposition}

We note that the rescaled ground state $\til{\phi}_{\om}=\til{f}_\om+\til{f}(0)\beta_{\al}(\om)^{-1}G_1$ satisfies the equation
\begin{equation}\label{eq:tilspf}
  (-\Del+1)\til{f}_\om
  -\til{\phi}_\om^{p}
  =0.
\end{equation}
Moreover, for $\om>\om_1$, where $\om_1$ is as in Proposition~\ref{prop:reg1}, the derivative $\pt_\om \til f_\om$ is in $H^2(\R^2)$, and $\pt_\om \til f_\om(0)$ makes sense for

We prepare some lemmas.

\begin{lemma}
For $\om>\om_1$, the following Pohozaev identity holds:
\begin{equation}\label{eq:Pohozaev1}
  \|\til{\phi}_{\om}\|_{L^2}^2
  =\frac{\til{f}_\om(0)^2}{4\pi\beta_\al(\om)^2}
  +\frac{2}{p+1}\|\til{\phi}_{\om}\|_{L^{p+1}}^{p+1}.
\end{equation}
In particular,
\begin{equation}\label{eq:Pohozaev2}
  \frac{d}{d\om}\|\til{\phi}_{\om}\|_{L^2}^2
  =\frac{\til{f}_\om(0)\pt_\om\til f(0)}{2\pi\beta_\al(\om)^2}
  -\frac{\til{f}_\om(0)^2}{8\pi^2\om\beta_\al(\om)^3}
  +2\int\til{\phi}_{\om}^p\pt_\om\til{\phi}_{\om}.
\end{equation}
\end{lemma}

\begin{proof}
By multiplying $x\cdot\nabla\til{\phi}_\om$ with the equation \eqref{eq:tilspf} and integrating it, we have
\begin{equation}\label{eq:Poho1}
  \langle(-\Del+1)\til{f}_\om,x\cdot\nabla\til{f}_\om\rangle
  +\frac{\til{f}_\om(0)}{\beta_\al(\om)}\langle(-\Del+1)\til{f}_\om,x\cdot\nabla G_1\rangle
  =\langle\til{\phi}_\om^p,x\cdot\nabla\til{\phi}_\om\rangle.
\end{equation}
From properties of the scaling, we have
\begin{align}
  \label{eq:Poho2}
  &\langle(-\Del+1)f,x\cdot\nabla f\rangle
  =\left.\frac12\frac{d}{d\lam}\|f(\lam\cdot)\|_{H^1}^2\right|_{\lam=1}
    =-\|f\|_{L^2}^2,
\\\label{eq:Poho3}
  &\langle\phi^p,x\cdot\nabla\phi\rangle
  =\left.\frac1{p+1}\frac{d}{d\lam}\|\phi(\lam\cdot)\|_{L^{p+1}}^{p+1}\right|_{\lam=1}
  =-\frac{2}{p+1}\|\phi\|_{L^{p+1}}^{p+1}.
\end{align}
Moreover, by Lemma~\ref{lem:xgG1} we have
\begin{equation}\label{eq:Poho4}
  \langle(-\Del+1)f,x\cdot\nabla G_1\rangle
  =-2(f,G_1)_{L^2}.
\end{equation}
By using \eqref{eq:Poho2}, \eqref{eq:Poho3}, \eqref{eq:Poho4}, we can rewrite \eqref{eq:Poho1} as
\begin{equation}\label{eq:Poho5}
  \|\til{f}_\om\|_{L^2}^2
  +2\frac{\til{f}_\om(0)}{\beta_\al(\om)}(\til{f}_\om,G_1)_{L^2}
  =\frac{2}{p+1}\|\til{\phi}_\om\|_{L^{p+1}}^{p+1}.
\end{equation}
Moreover, from the expression
\[\|\til{\phi}_\om\|_{L^2}^2
  =\left\|\til{f}_\om+\frac{\til{f}_\om(0)}{\beta_\al(\om)}G_1\right\|_{L^2}^2
  =\|\til{f}_\om\|_{L^2}^2
  +2\frac{\til{f}_\om(0)}{\beta_\al(\om)}(\til{f}_\om,G_1)_{L^2}
  +\frac{\til{f}_\om(0)^2}{4\pi\beta_\al(\om)^2}, \]
we obtain \eqref{eq:Pohozaev1}. By differentiating \eqref{eq:Pohozaev1} we have \eqref{eq:Pohozaev2}.
\end{proof}

\begin{lemma}\label{lem:pdeltilphi}
For $\om>\om_1$,
\[(p-1)\int\til{\phi}_\om^{p}\pt_\om\til{\phi}_\om
  =\frac{\til{f}_\om(0)^2}{4\pi\om\beta_\al(\om)^2}. \]
\end{lemma}

\begin{proof}
By the equation \eqref{eq:tilspf}, we have $\langle(-\Del+1)\til{f}_\om-\til{\phi}_{\om}^p,\pt_\om\til{\phi}_\om\rangle
=0$. This can be rewritten from the expression
\begin{equation}\label{eq:delomphitil}
  \pt_\om\til{\phi}_\om
  =\pt_\om\til{f}_\om
  +\frac{\pt_\om\til{f}_\om(0)}{\beta_\al(\om)}G_1
  -\frac{\til{f}_\om(0)}{4\pi\om\beta_\al(\om)^2}G_1
\end{equation}
as
\begin{align}
  \label{eq:intphip1}
  \int\til{\phi}_{\om}^p\partial_\om\til{\phi}_{\om}
  &=\langle(-\Del+1)\til{f}_\om,\pt_\om\til{f}_\om
  +\frac{\pt_\om\til{f}_\om(0)}{\beta_\al(\om)}G_1
  -\frac{\til{f}_\om(0)}{4\pi\om\beta_\al(\om)^2}G_1\rangle
\\\notag
  &=\langle(-\Del+1)\til{f}_\om,\pt_\om\til{f}_\om\rangle
  +\frac{\til{f}_\om(0)\pt_\om\til{f}_\om(0)}{\beta_\al(\om)}
  -\frac{\til{f}_\om(0)^2}{4\pi\om\beta_\al(\om)^2},
\end{align}
where we used the fact that $G_1$ is a solution of $(-\Del+1)G_1=\delta_0$.

By differentiating the equation \eqref{eq:tilspf} with respect to $\om$, we have
\begin{equation}\label{eq:spdelo}
  (-\Del+1)\pt_\om\til{f}_\om
  -p\til{\phi}_\om^{p-1}\pt_\om\til{\phi}_\om=0.
\end{equation}
By multiplying this equation with $\til{\phi}_\om$ and integrating it, we have
\begin{align}\label{eq:intphip2}
  p\int\til{\phi}_\om^{p}\pt_\om\til{\phi}_\om
  &=\langle(-\Del+1)\pt_\om\til{f}_\om,\til{\phi}_\om\rangle
\\\notag
  &=\langle(-\Del+1)\pt_\om\til{f}_\om,\til{f}_\om\rangle+\frac{\til{f}_\om(0)\pt_\om\til{f}_\om(0)}{\beta_\al(\om)}.
\end{align}
Therefore, the assertion follows from \eqref{eq:intphip1} and \eqref{eq:intphip2}.
\end{proof}

\begin{lemma}\label{lem:deltilf0}
There exists $C>0$ such that for $\om>\om_1$,
\begin{align}
  \label{eq:deltilf0}
  |\pt_\om\til{f}_\om(0)|
  &\le \frac{C}{\om\beta_\al(\om)^{3/2}}.
\end{align}
\end{lemma}

\begin{proof}
Noting that $\beta_{\til\al}(1)=\beta_{\al}(\om)$, we have the relation
\[(-\Del+1)\pt_\om\til{f}_\om
  =(-\Del_{\til{\al}}+1)\left(\pt_\om\til{f}_\om+\frac{\pt_\om\til{f}_\om(0)}{\beta_\al(\om)}G_1\right). \]
Therefore, by the expression~\eqref{eq:delomphitil}, we can rewrite \eqref{eq:spdelo} as
\begin{align*}
  \til{L}_{\om}\left(\pt_\om\til{f}_\om+\frac{\pt_\om\til{f}_\om(0)}{\beta_\al(\om)}G_1\right)
  =-\frac{p\til{f}_\om(0)}{4\pi\om\beta_\al(\om)^2}\til{\phi}_\om^{p-1}G_1.
\end{align*}
Thus, by the invertibility of $\til{L}_{\om}$ in the radial space (Lemma~\ref{lem:ftL}), we obtain
\[\pt_\om\til{f}_\om+\frac{\pt_\om\til{f}_\om(0)}{\beta_\al(\om)}G_1
  =-\frac{p\til{f}_\om(0)}{4\pi\om\beta_\al(\om)^2}(\til{L}_{\om})^{-1}(\til{\phi}_\om^{p-1}G_1). \]

From this expression and the definition of the bilinear form for $-\Del_\al$, we have the estimate
\begin{align}
  \label{eq:deltilf0est1}
  |\pt_\om\til{f}_\om(0)|
  &=\biggl|\biggl\langle(-\Del_{\til{\al}}+1)G_1,\pt_\om\til{f}_\om+\frac{\pt_\om\til{f}_\om(0)}{\beta_\al(\om)}G_1\biggr\rangle\biggr|
\\&\notag=\frac{p\til{f}_\om(0)}{4\pi\om\beta_\al(\om)^2}\left|\left\langle(-\Del_{\til{\al}}+1)G_1,(\til{L}_{\om})^{-1}(\til{\phi}_\om^{p-1}G_1)\right\rangle\right|
\\&\notag\le\frac{p\til{f}_\om(0)}{4\pi\om\beta_\al(\om)^2}\|(-\Del_{\til{\al}}+1)G_1\|_{\til{H}_{\al,\om}^{-1}}\|(\til{L}_{\om})^{-1}(\til{\phi}_\om^{p-1}G_1)\|_{\til{H}_{\al,\om}^1}.
\end{align}
A Direct calculation gives
\begin{align}
  \label{eq:deltilf0est2}
  \|(-\Del_{\til{\al}}+1)G_1\|_{\til{H}_{\al,\om}^{-1}}
  &=\sup\{\langle(-\Del_{\til{\al}}+1)G_1,w\rangle \colon\, \|w\|_{\til{H}_{\al,\om}^1}\le 1\}
\\\notag&=\sup\{d\beta_\al(\om)\colon\, |d|\le\beta_\al(\om)^{-1/2}\}
  =\beta_\al(\om)^{1/2}.
\end{align}
From Lemma~\ref{lem:ftL} we obtain
\begin{equation}\label{eq:deltilf0est3}
  \begin{aligned}
  \|(\til{L}_{\om})^{-1}(\til{\phi}_\om^{p-1}G_1)\|_{\til{H}_{\al,\om}^1}
  &\lesssim\|\til{\phi}_\om^{p-1}G_1\|_{\til{H}_{\al,\om,\rad}^{-1}}
  =\sup\{\langle\til{\phi}_\om^{p-1}G_1,w\rangle \colon\, \|w\|_{\til{H}_{\al,\om,\rad}^1}\le 1\}
\\&\le\sup\Bigl\{\langle\til{\phi}_\om^{p-1}G_1,g+dG_1\rangle \colon\, \substack{\|g\|_{H_\rad^1}\le 1,\\ d^2\beta_\al(\om)\le1}\Bigr\}
\\&\lesssim\|\til{\phi}_\om\|_{L^{p+1}}^{p-1}\|G_1\|_{L^{p+1}}
  +\frac{1}{\beta_\al(\om)^{1/2}}\|\til{\phi}_\om\|_{L^{p+1}}^{p-1}\|G_1\|_{L^{p+1}}^2.
  \end{aligned}
\end{equation}
Combining the estimates \eqref{eq:deltilf0est1}, \eqref{eq:deltilf0est2}, and \eqref{eq:deltilf0est3} and using the bounds from Proposition~\ref{lem:4.1}, we obtain the conclusion.
\end{proof}

\begin{proof}[Proof of Proposition~\ref{prop:stabinsta}]
By the definition of the scaling $\til{\phi}_\om$, we have
\begin{align*}
  \|\phi_\om\|_{L^2}^2
  =\om^{(3-p)/(p-1)}\|\til{\phi}_\om\|_{L^2}^2.
\end{align*}
By \eqref{eq:Pohozaev2} and Lemma~\ref{lem:pdeltilphi}, we have
\begin{align*}
  \frac{d}{d\om}\|\phi_\om\|_{L^2}^2
  &=\om^{2(2-p)/(p-1)}\left(\frac{3-p}{p-1}\|\til{\phi}_\om\|_{L^2}^2
  +\om\frac{d}{d\om}\|\til{\phi}_\om\|_{L^2}^2\right)
\\&=\om^{2(2-p)/(p-1)}\biggl(
  \begin{aligned}[t]
  &\frac{3-p}{p-1}\|\til{\phi}_\om\|_{L^2}^2
  +\frac{\om\til{f}_\om(0)\partial_\om\til{f}_\om(0)}{2\pi\beta_\al(\om)^2}
\\&\quad-\frac{\til{f}_\om(0)^2}{8\pi^2\beta_\al(\om)^3}
  +2\om\int\til{\phi}_{\om}^p\partial_\om\til{\phi}_{\om}\biggr)
  \end{aligned}
\\&=\om^{2(2-p)/(p-1)}\biggl(\begin{aligned}[t]
  &\frac{3-p}{p-1}\|\til{\phi}_\om\|_{L^2}^2
  +\frac{\om\til{f}_\om(0)\partial_\om\til{f}_\om(0)}{2\pi\beta_\al(\om)^2}
\\&\quad-\frac{\til{f}_\om(0)^2}{8\pi^2\beta_\al(\om)^3}
  +\frac{\til{f}_\om(0)^2}{2(p-1)\pi\beta_\al(\om)^2}\biggr).
  \end{aligned}
\end{align*}
Note that Lemma~\ref{lem:deltilf0} implies
\[\left|\frac{\om\til{f}_\om(0)\partial_\om\til{f}_\om(0)}{2\pi\beta_\al(\om)^2}\right|
  \lesssim \frac{1}{\beta_\al(\om)^{7/2}}. \]
Therefore, we obtain
\[\frac{d}{d\om}\|\phi_\om\|_{L^2}^2
  =\om^{2(2-p)/(p-1)}\left(\frac{3-p}{p-1}\|\til{\phi}_\om\|_{L^2}^2
  +\frac{\til{f}_\om(0)^2}{2(p-1)\pi\beta_\al(\om)^2}
  +O\left(\frac{1}{\beta_\al(\om)^{3}}\right)\right) \]
as $\om\to\infty$. By Proposition~\ref{lem:4.1}, $\|\til\phi_\om\|_{L^2}$ and $\til{f}_\om(0)$ converge to positive constants as $\om\to\infty$. Therefore, we deduce the conclusion.
\end{proof}

\begin{proof}[Proof of Theorem~\ref{thm:stablarge}]
The assertion follows from Propositions~\ref{prop:stabinsta} and \ref{prop:scstabinsta}. This completes the proof.
\end{proof}

    \appendix
    \section{Review of the properties of Laplace operator with point interaction}\label{sec:A}

Let us review of the properties of the operator $-\Del_\al$. An important feature of the family $-\Del_{\al}$ with $\al\in\R$ is the following explicit formula for the resolvent, valid for every $\lam >0$.
\begin{equation}\label{eq:res_formula}
  (-\Del_\al+\lam)^{-1}g
  =(-\Del+\lam)^{-1}g
  +\frac{(g,G_\lam)_{L^2}}{\beta_{\al}(\lam) }G_\lam .
\end{equation}
Identity \eqref{eq:res_formula} says that the resolvent of $-\Del_\al$ is a rank-one perturbation of the free resolvent. As a consequence, it is possible to deduce the spectral properties \eqref{eq:spectrum1} and \eqref{eq:negaengen} of $-\Del_\al$.



One can apply Strichartz estimates for the non-negative self-adjoint operator $-\Del_\al - e_\al$ since \cite[Theorem~1.3]{CMY19} guarantees the existence of wave operators
\[
W_{\pm}=\lim _{t \rightarrow \pm \infty} e^{i t \Del_{\al}} e^{-i t \Del}
\]
in $L^p(\R^2)$ for $1<p<\infty$. We know also that $W_\pm$ are  complete in the sense that ran $W_{\pm}=L_{\mr{ac}}^{2}(-\Del_\al)$, the absolutely continuous subspace of $L^{2}(\R^{2})$ for $-\Del_\al$. In our case this is the space
\[L_{\mr{ac}}^{2}(-\Del_\al)
  =\{f\in L^2(\R^2)\colon\,
  (f , \chi_\al)_{L^2}=0\},
  \]
so we have
\begin{align*}
  &W_{\pm}^{*} W_{\pm} =1, &
  &W_{\pm} W_{\pm}^{*} = P_{\mr{ac}} (-\Del_\al),
\end{align*}
where $P_{\mr{ac}} (-\Del_\al)$ is the orthogonal projection onto $L_{\mr{ac}}^{2} (-\Del_\al)$. The wave operators satisfy the intertwining property
\[f(-\Del_\al) P_{\mr{ac}}(-\Del_\al)
  =W_{\pm}f(-\Del)W_{\pm} \]
for any Borel function $f$ on $\R$.

By using the intertwining property one can deduce the following Strichartz estimate \cite[Corollary~1.5]{CMY19}:
\begin{equation}\label{eq.STr1}
\|e^{i t\Del_{\al}}P_{ac}(-\Del_\al) f\|_{L_t^r(\R, L_x^q)}\lesssim\|f\|_{L^2},
\end{equation}
where $(r,q)$ is an admissible  Strichartz pair, i.e.
\begin{equation}\label{eq.STr2}
  1 < r \leq \infty, \quad
  1 < q < \infty, \quad
  \frac{1}{r}+ \frac{1}q = \frac{1}{2}.
\end{equation}
Since the orthogonal projection on $L_{\mr{ac}}^{2}(-\Del_{\al})$ is given by
\[P_{\mr{ac}}(-\Del_\al)f
  =f-(f, \chi_\al)_{L^2}\chi_\al, \]
we see that
\[e^{i t\Del_{\al}} f
  =e^{i t\Del_{\al}}P_{ac}(-\Del_\al) f
  +(f, \chi_\al)_{L^2}e^{i t e_\al}\chi_\al \]
for all $f \in L^2(\R^2)$. So the property \eqref{eq.pprlm1} guarantees that we have
the following (local in time)  Strichartz estimate: there  exists a constant $C>0$ such that for any $T \in (0,1]$ we have
\begin{equation}\label{eq.STr3}
  \|e^{i t\Del_{\al}} f\|_{L_t^r([0,T], L_x^q)}\le
  C\|f\|_{L^2}
\end{equation}
for all $f \in L^2(\R^2)$. By using $TT^*$ argument and Christ--Kiselev lemma we arrive at the following Strichartz estimate: there exists a constant $C>0$ so that for any $T \in (0,1]$ we have
\begin{equation}\label{eq.STr4}
  \biggl\|\int_0^t e^{i(t-s)\Del_{\al}}F(s)\,ds\biggr\|_{L_t^{r_1}([0,T], L_x^{q_1})}
  \le C\|F\|_{L^{r_2'}([0,T], L^{q_2'})}
\end{equation}
for any  $ F \in L^{r_2'}([0,T], L^{p_2'}(\R^2)).$
Here and below $(r_1,q_1)$ and $(r_2,q_2)$ are admissible Strichartz pairs, i.e.~$\frac{1}{r_j}+\frac{1}{q_j}=\frac{1}{2}$, $q_j\in[2,\infty)$, for $j=1,2$.

\begin{remark}
The $L^{1}$-$L^{\infty}$ dispersive estimates cannot hold, in fact even for a smooth initial data $f$ the evolution $e^{it\Del_{\al}}f$ exhibits, for almost every time $t\neq 0$, a non-trivial singular component proportional to $G_{\lam}\not\in L^{\infty}(\R^2)$.
\end{remark}

    \section{Local Well-posedness in $H_\al^1(\R^2)$} \label{sec:B}

In this section, we establish the local well-posedness in the energy space in $H_\al^1(\R^2)$. To this aim, we apply the abstract theory of Okazawa, Suzuki, and Yokota~\cite{OSY12} to construct a weak solution to \eqref{NLS} with initial data $u_0\in H_\al^1(\R^2)$. Then we establish the uniqueness of the solution by using the Strichartz estimate obtained by \cite{CMY19}.

First, we construct  a weak solution to \eqref{NLS} by using \cite[Theorem~2.2]{OSY12}.

\begin{lemma}\label{lem:constructsol}
For any $M>0$ there exists $T_M>0$ such that the following is true. For $u_0\in H_\al^1(\R^2)$ with $\|u_0\|_{H_\al^1}\le M$, there exists a local weak solution
\[u\in C_\mathrm{w}([-T_M,T_M],H_\al^1(\R^2))\cap W^{1,\infty}(-T_M,T_M;H_\al^{-1}(\R^2)) \]
of \eqref{NLS} satisfying
\begin{align*}
  \|u(t)\|_{L^2}
  =\|u_0\|_{L^2},\quad
  E(u(t))\le E(u_0)
\end{align*}
for all $t\in[-T_M,T_M]$.
\end{lemma}

\begin{proof}
We will apply \cite[Theorem~2.2]{OSY12} as
\begin{align*}
  S&=-\Del_\al-e_\al,
\\X&=L^2(\R^2),\quad
  X_S=H_\al^1(\R^2),\quad
  X_S^*=H_\al^{-1}(\R^2)
\\g(v)
  &=e_\al v-|v|^{p-1}v.
\end{align*}
Under this setting, we see that $S$ is a nonnegative self-adjoint operator in $L^2(\R^2)$ and that $H_\al^1(\R^2)=D((1+S)^{1/2})$. After that, we only have to verify \cite[\textbf{(G1)}--\textbf{(G5)}]{OSY12} given as follows.

\noindent\textbf{(G1)}: there exists $G\in C^1(X_S,\R)$ such that $G'=g$.

\noindent\textbf{(G2)}: for all $M>0$ there exists $C(M)>0$ such that
\[\|g(u)-g(v)\|_{X_S^*}
  \le C(M)\|u-v\|_{X_S}\quad
  \forall u,v\in X_S
  \text{ with $\|u\|_{X_S}$,\,$\|v\|_{X_S}\le M$}. \]

\noindent\textbf{(G3)}: for all $M,\delta>0$ there exists $C_\delta(M)>0$ such that
\[|G(u)-G(v)|
  \le \delta
  +C(M)\|u-v\|_X\quad
  \forall u,v\in X_S
  \text{ with $\|u\|_{X_S}$,\,$\|v\|_{X_S}\le M$}. \]

\noindent\textbf{(G4)}:
\[\langle g(u),iu\rangle_{X_S^*,X_S}
  =0\quad
  \forall u\in X_S. \]

\noindent\textbf{(G5)}:
given a bounded open interval $I\subset\R$, let $(w_n)_{n\in\N}$ by any bounded sequence in $L^\infty(I,X_S)$ such that
\[\left\{\begin{aligned}
  w_n(t)&\to w(t)\ (n\to\infty)
  &&\text{weakly in $X_S$ a.a. $t\in I$},
\\g(w_n)&\to f\ (n\to\infty)
  &&\text{weakly$^*$ in $L^\infty(I,X_S^*)$}.
\end{aligned}\right. \]
Then
\[\int_I\langle f(t),iw(t)\rangle_{X_S^*,X_S}\,dt
  =\lim_{n\to\infty}\int_I\langle g(w_n(t)),iw_n(t)\rangle_{X_S^*,X_S}\,dt. \]

Now we check \textbf{(G1)}--\textbf{(G5)}.

The conditions \textbf{(G1)} are easily verified as
\[G(v)
  =\frac{1}{p+1}\|v\|_{L^{p+1}}^{p+1},\quad
  v\in H_\al^1(\R^2) \]
by standard inequalities and the embedding $L^q(\R^2)\subset H_\al^1(\R^2)$ obtained in \eqref{eq.sob1}. Similarly, the condition \textbf{(G2)} also can be  verified.

The condition \textbf{(G3)} follows from the following estimate:
\begin{align*}
  |G(u)-G(v)|
  &\lesssim\int(|u|^p+|v|^p)|u-v|\,dx
\\&\lesssim(\|u\|_{L^{2p}}^p+\|v\|_{L^{2p}}^p)\|u-v\|_{L^2}
\\&\lesssim M^p\|u-v\|_{L^2}
\end{align*}
for $u,v\in H_\al^1(\R^2)$ with $\|u\|_{H_\al^1},\,\|v\|_{H_\al^1}\le M$.

The conditions \textbf{(G4)} is clear from the definition of $g$.

Finally, we will check the conditions \textbf{(G5)}. From \cite[Lemma~5.3]{OSY12}, it is enough to show that if $(u_n)_{n\in\N}$ is a sequence $H_\al^1(\R^2)$ satisfies
\[\left\{\begin{aligned}
  u_n&\to u\ (n\to\infty)
  &&\text{weakly in $H_\al^1(\R^2)$},
\\g(u_n)&\to f\ (n\to\infty)
  &&\text{weakly in $H_\al^{-1}(\R^2)$},
\end{aligned}\right. \]
then $f=g(u)$.

We follow the argument in \cite[Proof of Theorem~1.1]{S15}. Let $\vp\in C_c^\infty(\R^2)$. Then from the weak convergence of $(u_n)_{n\in\N}$ in $H_\al^1(\R^2)$ and the compactness $L_{\mr{loc}}^{p+1}(\R^2)\hookrightarrow H_{\mr{loc}}^1(\R^2)$ we see that
\[u_n\to u\quad
  \text{in }L_{\mr{loc}}^{p+1}(\R^2). \]
Thus,
\begin{align*}
  |\langle g(u_n)-g(u),\vp\rangle|
  &\lesssim\|\vp\|_{L^{p+1}}(\|u_n\|_{L^{p+1}}^{p-1}+\|u\|_{L^{p+1}}^{p-1})\|u_n-u\|_{L^{p+1}(\operatorname{supp}\vp)}
  \to 0
\end{align*}
as $n\to\infty$. This means that $g(u_n)\to g(u)$ in $\ml{D}'(\R^2)$. On the other hand, $g(u_n)\to f$ in $H_\al^{-1}(\R^2)$ and hence in $\ml{D}'(\R^2)$. Therefore we obtain $f=g(u)$. Thus, \textbf{(G5)} is verified.

We have just finished the verification of \textbf{(G1)}--\textbf{(G5)}. Therefore, \cite[Theorem~2.2]{OSY12} implies the conclusion.
\end{proof}

\begin{lemma}\label{lem:uniquenesssol}
Let $u_0\in H_\al^1(\R^2)$. If $u_1,u_2\in L^\infty(-T,T;H_\al^1(\R^2))$ are two weak solutions of \eqref{NLS} with $u_1(0)=u_2(0)=u_0$, then $u_1=u_2$.
\end{lemma}

\begin{proof}
Without loss of generality we can assume $T \in (0,1]$. Let
\[r = r(p) \ce \frac{2(p+1)}{p+1}. \]
Then $(r,p+1)$ is a admissible pair.  By the Strichartz estimate \eqref{eq.STr4} for
\[ u_j(t) = e^{i t \Del_\al}u_0
  -i\int_0^t e^{i(t-\tau)\Del_\al} |u_j(\tau)|^{p-1} u_j(\tau)\, d \tau, \]
we see that
\[\|u_1 - u_2\|_{L^r([0,T], L^{p+1}_x)}
  \lesssim \| |u_1|^{p-1}u_1 - |u_2|^{p-1}u_2 \|_{L^{1}([0,T],L^2_x)}. \]
From
\[\| |u_1|^{p-1}u_1- |u_2|^{p-1}u_2 \|_{L^2_x}
  \lesssim (\|u_1\|_{L^{2(p+1)}_x}^{p-1} + \|u_2\|_{L^{2(p+1)}_x}^{p-1} )\|u_1-u_2\|_{L^{p+1}_x}
  \]
and the Sobolev inequality \eqref{eq.sob1} we deduce
\[\|u_j\|_{L^{2(p+1)}_x}
  \lesssim \|u_j\|_{H^1_\al(\R^2)} \lesssim 1. \]
Hence
\begin{align*}
   \|u_1  - u_2\|_{L^4([0,T], L^4_x)}
   &\le C \int_0^T \|u_1(\tau)-u_2(\tau)\|_{L^{p+1}_x} d\tau
\\ &\le C T^{(p+1)/2p}\|u_1  - u_2\|_{L^r([0,T], L^{p+1}_x)}
\end{align*}
so with $T$ sufficiently small so that $CT^{(p+1)/(2p)} < 1$ we conclude that $u_1(t)=u_2(t)$ for $t \in [0,T]$. In a similar way, we obtain  $u_1(t)=u_2(t)$ for $t \in [-T,0]$.
\end{proof}

\begin{proof}[Proof of Proposition~\ref{Prop:LWP}]
The assertion follows from Lemmas~\ref{lem:constructsol}, \ref{lem:uniquenesssol}, and \cite[Theorem~2.3]{OSY12}.
\end{proof}

    \section*{Acknowledgements}

NF was supported by JSPS KAKENHI Grant Number JP20K14349.
VG was partially supported by Project 2017 ``Problemi stazionari e di evoluzione nelle equazioni di campo nonlineari'' of INDAM, GNAMPA - Gruppo Nazionale per l'Analisi Matematica, la Probabilita e le loro Applicazioni, by Institute of Mathematics and Informatics, Bulgarian Academy of Sciences, by Top Global University Project, Waseda University and the Project PRA 2018 49 of University of Pisa.
MI is supported by JST CREST Grant Number JPMJCR1913, Japan and Grant-in-Aid for Young Scientists Research (No.19K14581), Japan Society for the Promotion of Science.

\bibliographystyle{amsplain_abbrev_nobysame_nonumber}
\bibliography{fgi2021arxiv1}

\end{document}